\documentclass[11pt]{amsart}

\DeclareMathOperator{\Lie}{Lie}

\DeclareMathOperator{\Char}{char}
\usepackage[usenames,dvipsnames]{xcolor}
\usepackage[normalem]{ulem}
\usepackage{amsmath,amssymb,amsthm,graphics,amscd,mathrsfs}
\usepackage{longtable}
\usepackage{cite}
\usepackage{graphicx}
\usepackage{epsf}
\usepackage{fancyhdr,hyperref}
\usepackage{lscape}
\usepackage{dynkin-diagrams}
\usepackage{caption}
\usepackage{mdframed}
\usepackage{tikz-cd}
\tikzset{
    labl/.style={anchor=south, rotate=90, inner sep=.5mm}
}

\usepackage{verbatim}

\hypersetup{colorlinks=true,citecolor=Purple}

  \renewenvironment{thebibliography}[1]{
    \begin{oldthebibliography}{#1}
      \setlength{\parskip}{0ex}
      \setlength{\itemsep}{0ex}
  }
  {
    \end{oldthebibliography}
  }
  
\setlongtables
\begin{document}

\newcounter{rownum}
\setcounter{rownum}{0}
\newcommand{\ab}{\addtocounter{rownum}{1}\arabic{rownum}}

\newcommand{\x}{$\times$}
\newcommand{\bb}{\mathbf}
\newcommand{\UU}{\mathscr{U}}
\newcommand{\CC}{\mathscr{C}}
\newcommand{\GG}{\mathscr{G}}
\newcommand{\BB}{\mathscr{B}}
\newcommand{\NN}{\mathscr{N}}
\newcommand{\XX}{\mathscr{X}}

\newcommand{\Ind}{\mathrm{Ind}}
\newcommand{\R}{\mathrm{R}}
\newcommand{\RR}{\mathscr{R}}
\providecommand{\G}{\mathscr{G}}
\newcommand{\hra}{\hookrightarrow}
\newcommand{\sss}{\mathrm{ss}}

\numberwithin{equation}{section}
\numberwithin{figure}{section}

\newtheorem{lemma}{Lemma}[section]
\newtheorem{theorem}[lemma]{Theorem}
\newtheorem*{TA}{Theorem A}
\newtheorem*{TB}{Theorem B}
\newtheorem*{TC}{Theorem C}
\newtheorem*{CorC}{Corollary C}
\newtheorem*{TD}{Theorem D}
\newtheorem*{TE}{Theorem E}
\newtheorem*{PF}{Proposition E}
\newtheorem*{C3}{Corollary 3}
\newtheorem*{T4}{Theorem 4}
\newtheorem*{C5}{Corollary 5}
\newtheorem*{C6}{Corollary 6}
\newtheorem*{C7}{Corollary 7}
\newtheorem*{C8}{Corollary 8}
\newtheorem*{claim}{Claim}
\newtheorem{cor}[lemma]{Corollary}
\newtheorem{conjecture}[lemma]{Conjecture}
\newtheorem{prop}[lemma]{Proposition}
\newtheorem{question}[lemma]{Question}
\theoremstyle{definition}
\newtheorem{example}[lemma]{Example}
\newtheorem*{idea}{Basic Idea}
\newtheorem{examples}[lemma]{Examples}
\newtheorem{algorithm}[lemma]{Algorithm}
\newtheorem*{algorithm*}{Algorithm}
\theoremstyle{remark}
\newtheorem{remark}[lemma]{Remark}
\newtheorem{remarks}[lemma]{Remarks}
\newtheorem{obs}[lemma]{Observation}
\theoremstyle{definition}
\newtheorem{defn}[lemma]{Definition}

  \def\hal{\unskip\nobreak\hfil\penalty50\hskip10pt\hbox{}\nobreak
  \hfill\vrule height 5pt width 6pt depth 1pt\par\vskip 2mm}

\renewcommand{\labelenumi}{(\roman{enumi})}
\newcommand{\Hom}{\mathrm{Hom}}
\newcommand{\Int}{\mathrm{int}}
\newcommand{\Ext}{\mathrm{Ext}}
\newcommand{\opH}{\mathrm{H}}
\newcommand{\D}{\mathscr{D}}
\newcommand{\soc}{\mathrm{Soc}}
\newcommand{\SO}{\mathrm{SO}}
\newcommand{\Sp}{\mathrm{Sp}}
\newcommand{\SL}{\mathrm{SL}}
\newcommand{\GL}{\mathrm{GL}}
\newcommand{\PGL}{\mathrm{PGL}}
\newcommand{\OO}{\mathcal{O}}
\newcommand{\Y}{\mathbf{Y}}
\newcommand{\QQ}{\mathsf{Q}}
\newcommand{\X}{\mathbf{X}}
\newcommand{\diag}{\mathrm{diag}}
\newcommand{\End}{\mathrm{End}}
\newcommand{\tr}{\mathrm{tr}}
\newcommand{\Stab}{\mathrm{Stab}}
\newcommand{\red}{\mathrm{red}}
\newcommand{\Aut}{\mathrm{Aut}}
\renewcommand{\H}{\mathcal{H}}
\renewcommand{\u}{\mathfrak{u}}
\newcommand{\Ad}{\mathrm{Ad}}
\newcommand{\N}{\mathcal{N}}
\newcommand{\Z}{\mathbb{Z}}
\newcommand{\la}{\langle}\newcommand{\ra}{\rangle}
\newcommand{\gl}{\mathfrak{gl}}
\newcommand{\g}{\mathfrak{g}}
\newcommand{\F}{\mathbb{F}}
\newcommand{\m}{\mathfrak{m}}
\renewcommand{\b}{\mathfrak{b}}
\newcommand{\p}{\mathfrak{p}}
\newcommand{\q}{\mathfrak{q}}
\renewcommand{\l}{\mathfrak{l}}
\newcommand{\del}{\partial}
\newcommand{\h}{\mathfrak{h}}
\renewcommand{\t}{\mathfrak{t}}
\renewcommand{\k}{\mathfrak{k}}
\newcommand{\Gm}{\mathbb{G}_m}
\renewcommand{\c}{\mathfrak{c}}
\renewcommand{\r}{\mathfrak{r}}
\newcommand{\n}{\mathfrak{n}}
\newcommand{\s}{\mathfrak{s}}
\newcommand{\Q}{\mathbb{Q}}
\newcommand{\z}{\mathfrak{z}}
\newcommand{\pso}{\mathfrak{pso}}
\newcommand{\so}{\mathfrak{so}}
\renewcommand{\sl}{\mathfrak{sl}}
\newcommand{\psl}{\mathfrak{psl}}
\renewcommand{\sp}{\mathfrak{sp}}
\newcommand{\Ga}{\mathbb{G}_a}
\newcommand{\barB}{\overline{B}}
\newcommand{\barm}{\overline{\mathfrak{m}}}
\newcommand{\MM}{\mathsf{M}}
\newcommand{\LL}{\mathsf{L}}

\newenvironment{changemargin}[1]{%
  \begin{list}{}{%
    \setlength{\topsep}{0pt}%
    \setlength{\topmargin}{#1}%
    \setlength{\listparindent}{\parindent}%
    \setlength{\itemindent}{\parindent}%
    \setlength{\parsep}{\parskip}%
  }%
  \item[]}{\end{list}}

\parindent=0pt
\addtolength{\parskip}{0.5\baselineskip}

\title{A construction of pseudo-reductive groups with non-reduced root systems}

\author[M.\  Bate]{Michael Bate}
\address
{Department of Mathematics,
University of York,
York YO10 5DD,
United Kingdom}
\email{michael.bate@york.ac.uk}
\author[G. R\"ohrle]{Gerhard R\"ohrle}
\address
{Fakult\"at f\"ur Mathematik,
Ruhr-Universit\"at Bochum,
D-44780 Bochum, Germany}
\email{gerhard.roehrle@rub.de}
\author[D.\ Sercombe]{Damian Sercombe}
\address
{Fakult\"at f\"ur Mathematik,
Ruhr-Universit\"at Bochum,
D-44780 Bochum, Germany}
\email{Damian.Sercombe@ruhr-uni-bochum.de}
\author{David I. Stewart}
\address{Department of Mathematics,
Alan Turing Building,
Manchester,
M13 9PL, UK}
\email{david.i.stewart@manchester.ac.uk}

\pagestyle{plain}
\begin{abstract}
We describe a straightforward construction of the pseudo-split absolutely pseudo-simple groups of minimal type with irreducible root systems of type $BC_n$; these exist only in characteristic $2$. We also give a formula for the dimensions of their irreducible modules.
\end{abstract}
\maketitle

\section{Introduction}
Let $k$ be an arbitrary field, and $G$ a smooth connected affine $k$-group which is pseudo-reductive; that is, the largest connected smooth normal unipotent $k$-subgroup $\RR_{u,k}(G)$ of $G$ is trivial. The classification of $G$ up to isomorphism in \cite{CGP15} and \cite{CP17} becomes exceedingly intricate when $k$ has characteristic $2$---which to a large extent can be traced back to degeneracies in commutator relations in split reductive $k$-groups. The purpose of this note is to simplify one aspect of the characteristic $2$ theory by giving a brief construction of the pseudo-split absolutely pseudo-simple groups of minimal type with root system of type $BC_n$. We generate these groups inside the Weil restrictions of rather special semidirect products with a property which only exists in characteristic $2$. The upshot is that we are able to establish the existence part of the classification theorem of such groups \cite[Thm.~9.8.6]{CGP15} without recourse to Weil's birational group laws. 

The structure of this paper is as follows. 

In Section \ref{setup}, we give a complete description of semidirect products $V\rtimes H$ of a split simple $k$-group $H$ and an irreducible $H$-module $V$ such that $V\rtimes H$ has a split simple subgroup $M$ with $MV=HV$ but 
the projection $M\to H$ is not an
isomorphism. This situation arises only in characteristic $2$. 

In Section \ref{defng} we expand on a method in \cite[\S9.1]{CGP15} that exhibits interesting pseudo-reductive subgroups of the Weil restrictions of $\SL_2$ and $\PGL_2$ across non-trivial purely inseparable extensions. We generate certain maximal rank subgroups of the Weil restrictions of the special semidirect products as described in Section \ref{setup}, whose root systems are of type $BC_n$. In more detail: we first generate an abstract subgroup of $k$-points from some specified root groups, together with the $k$-points of the normaliser of a Cartan subgroup. We observe that this subgroup admits a $(B,N)$-pair, in the terminology of Tits. We then compare the `big cells' of the abstract group and its closure, which turns out to be enough to determine the root groups of its closure and conclude the abstract group \emph{is} the set of $k$-points of its closure. Then we deviate a little from \cite{CGP15} and use some standard results about groups with $(B,N)$-pairs to infer pseudo-reductivity or absolute pseudo-simplicity of these groups. (The $n=2$ case requires extra work.) 

Finally in Section \ref{irreps}, we give a formula for the dimensions of the irreducible modules of the groups of interest; this is explicit if $n=1$ or $2$, and otherwise reduces the problem to calculating the dimensions of the irreducible modules of $\Sp_{2n}$ in characteristic $2$. We also highlight Theorem \ref{prodsofrk1}, which gives a formula for the dimensions of an irreducible module of a pseudo-reductive group that is a product of commutative pseudo-split groups of rank $1$, in terms of the weight and the minimal fields of definition of the geometric unipotent radicals of the factors.

\subsection*{Acknowledgements} We thank the referees for alerting us to some errors in the first version and suggesting improvements to the exposition. The fourth author is supported by a Leverhulme Trust Research
Project Grant RPG-2021-080. 

\section{Set-up}\label{setup}
Let $k$ be a field and $G$ a smooth affine $k$-group (scheme); so $G$ is a functor from $k$-algebras to groups, represented by a finitely presented $k$-algebra $k[G]$. We say $G$ acts on an affine $k$-scheme $S$ if there is a natural transformation of functors $\alpha:G\times S\to S$ such that $\alpha(A):G(A)\times S(A)\to S(A)$ is a group action for any $k$-algebra $A$. All actions will be left actions. From a finite-dimensional $k$-vector space $V$ one gets a $k$-vector group $\underline{V}$ with $\underline{V}(A)=A\otimes_k V$. By a module for $G$, we mean a $k$-vector space $V$ and an action of $G$ on $\underline{V}$ such that for each $k$-algebra $A$, the group of $A$-points $G(A)$ acts $A$-linearly on $\underline{V}(A)$. If $V$ is a finite-dimensional $G$-module then one can form the semidirect product $\underline{V}\rtimes G$---another $k$-group; we abuse notation and write $V\rtimes G$. If $G$ is generated by the smooth closed subgroups $N$ and $H$, with $N$ normal in $G$ and $N\cap H=1$ (scheme-theoretic intersection), then $G\cong N\rtimes H$ where $H$ acts on $N$ by conjugation. We write ${}^hn=hnh^{-1}$ for $h\in H(A)$, $n\in N(A)$.

\subsection{Abstract complements in semidirect products}
In this subsection and the next we establish the following:

\begin{prop}\label{sdpprop} Let $V = k^{2n}$ be the natural $\Sp_{2n}$-module and let $\rho:V \rtimes \Sp_{2n} \to \Sp_{2n}$ be the natural projection. If $\Char(k)=2$ then there exists a $k$-subgroup $\SO_{2n+1} \hookrightarrow V \rtimes \Sp_{2n}$ that maps onto $\Sp_{2n}$ via $\rho$. This construction is unique in the following sense.

Let $L$ be a split simple $k$-group, let $U$ be an irreducible $L$-module and let $M$ be a split simple $k$-subgroup of $U\rtimes L$ such that the natural projection $\pi:U \rtimes L \to L$ sends $M$ onto $L$. Then either $M \cong L$ via $\pi$, or $\Char(k)=2$ and there exists an isomorphism $\psi:U\rtimes L \to V \rtimes \Sp_{2n}$ for some $n \in \mathbb N$ such that the following diagram commutes:
\[
\begin{tikzcd}
M \arrow[hookrightarrow]{r}{} \ar[d, "\sim" labl, "\psi|_M"] & U \rtimes L \arrow{r}{\pi} \ar[d, "\sim" labl, "\psi"] & L \ar[d, "\sim" labl, "\psi|_L"]  \\
\SO_{2n+1} \arrow[hookrightarrow]{r}{} & V \rtimes \Sp_{2n} \arrow{r}{\rho} & \Sp_{2n}
\end{tikzcd}
\]
The inclusion of $\SO_{2n+1}$ is the one specified above.
\end{prop}
We will show later in Corollary \ref{conjcor} that $M$ is unique in $Q:=U \rtimes L$ up to $Q(k)$-conjugacy.

We start by establishing the existence part of Proposition \ref{sdpprop}.

Let $k$ be a field of characteristic $2$ and consider the split simple $k$-group $M=\SO_{2n+1}$ in its action on its natural $(2n+1)$-dimensional module $X$, stabilising the standard quadratic form $q=x_{0}^2 + x_1x_2 + \dots + x_{2n-1}x_{2n}$, with associated bilinear form $(v,w)=q(v+w)-q(v)-q(w)$. Since $p=2$, we have $(v,v)=4q(v)-2q(v)=0$ so this form is alternating and so cannot be non-degenerate on an odd-dimensional space. Moreover, $M$ centralises the $1$-dimensional radical $k\cdot x$ of the form $(\cdot,\cdot)$,
which gives an embedding $M\subset P\subset \SL(X)$, where $P:=\Stab_{\SL(X)}(k\cdot x)$ is a maximal parabolic subgroup of $\SL(X)$. Choosing a basis of $X$, one gets a Levi decomposition $P \cong \RR_u(P)\rtimes\GL(U)$, where $U:=X/(k\cdot x)$.
The induced action of $M$ on $U$ preserves a non-degenerate alternating form, so for dimension reasons, the image of $M$ in $\GL(U)$ is isomorphic to $L:=\Sp_{2n}$. It is easy to see that $\RR_u(P)$ is a vector group and the conjugation action of $\GL(U)$ on $\RR_u(P)$ furnishes it with the structure of  
$(\det^{-1})\otimes U^*$ as a $\GL(U)$-module, where $\det$ is the determinant representation and $U$ is the $2n$-dimensional natural module. Since $L$ preserves a non-degenerate bilinear form on $U$, the restriction of $U$ to $L$ is self-dual. Also, the determinant representation of $L$ is trivial, so $\RR_u(P)$ is $L$-equivariantly isomorphic to $U$;  abusing notation slightly, we write $M\subseteq U\rtimes L$. Since $M$ surjects onto $L$, we must have $UM=U\rtimes L$. This establishes the existence part of Proposition \ref{sdpprop}. Note also that under the identifications made in this paragraph, $L$ and $M$ share a common ``diagonal'' maximal torus.

Now consider a triple of $k$-groups $(L,M,U)$ that satisfies the hypotheses of Proposition \ref{sdpprop} and assume $M\to L$ is not an isomorphism. 
Since the quotient map $\pi:U\rtimes L\to L$ does not restrict to an isomorphism on $M$, we conclude that the scheme-theoretic intersection $M\cap U\neq 1$ and so $M$ has a non-trivial unipotent normal subgroup (the scheme-theoretic kernel of $\pi|_M$). 
It was proved independently in \cite[Lem.~2.2]{PY06} and \cite[Thm.~2.2]{Vas05} that $\SO_{2n+1}$ over $k$ of characteristic $2$ is the only split simple group with this property, and that the non-trivial normal unipotent subgroup in question is unique.
Precisely:
 
\begin{lemma}\label{lem:lemma2.2} The unique non-trivial normal unipotent subgroup of $\SO_{2n+1}$ is the height $1$ subgroup scheme isomorphic to $\alpha_2^{2n}$, which corresponds to the ideal of short root spaces in the Lie algebra, and can be characterised as the direct product of the kernels of the Frobenius map on the short root groups of $\SO_{2n+1}$.\label{height1}\end{lemma}

Let $D$ be a split maximal torus of $L$, and denote $Q:=U\rtimes L$. Since $U$ is unipotent, $D$ remains a split maximal torus of $Q$. Being of the same rank, any split maximal torus of $M$ is also one of $Q$. Since split maximal tori of $Q$ are conjugate, we may replace $M$ by a $Q(k)$-conjugate to assume $D\subset M \cap L$. Now by the irreducibility of $U$ as an $L$-module (and hence as an $M$-module), $\Lie(\ker\pi|_M)=\Lie(U)$ and so, by the uniqueness in Lemma \ref{lem:lemma2.2}, the set of weights of $D$ on $U$ is the set of short roots in $M$ without multiplicity and $L = \Sp_{2n}$ as a quotient of $M = \SO_{2n+1}$. It follows that there is only one possibility for $U$ up to isomorphism as an $M$-module, and hence $U$ must be the natural module for $M$ and $L$.
This gives rise to the map $\psi$ in Proposition \ref{sdpprop}.
In order to complete the proof of the Proposition, we just need to establish that the restriction of $\psi$ to $M$ gives the isomorphism claimed.
This part is delayed to the next section, where we show the stronger result Corollary \ref{conjcor}.

\begin{remarks}(i). The basic thrust of these observations probably goes back to \cite[Prop.~1.5]{LS96}, where $M$ is described as a complement to $U$ in $U\rtimes L$ on the basis that $M(k)U(k)=L(k)U(k)$ and $M(k)\cap U(k)=1$. The more expedient (scheme-theoretic) notion of a complement is given in \cite[\S4]{SteNon} which recovers the natural correspondence between the space $\opH^1(L,U)$ and the set of $U(k)$-conjugacy classes of complements to $U$ in $U\rtimes L$.

(ii). Unsurprisingly, one finds a closely related cohomological phenomenon: if $G$ is a split simple $k$-group and $V$ is an irreducible $G$-module with $V^{[1]}$, then one gets a natural map $\opH^1(G,V)\to \opH^1(G,V^{[1]})$ which is an isomorphism unless $p=2$, $G\cong \Sp_{2n}$ and $V=L(\varpi_1)$ is the natural module for $G$ \cite{CPS83}. In the latter case $\opH^1(G,V)=0$ and $\opH^1(G,V^{[1]})\cong k$. This is essentially the same fact that for the first Frobenius kernel $G_1$ of a split simple $k$-group $G$, one has $\opH^1(G_1,k)=0$ unless $p=2$ and $G$ is of type $C_n$, in which case $\opH^1(G_1,k)$ is isomorphic as a $G$-module to a Frobenius twist of the natural module $V^{[1]}$; see \cite[II.12.2]{Jan03}.

(iii). An important phenomenon in the analysis of the subgroup structure of reductive groups over algebraically closed fields $k$ of characteristic $2$ is the occurrence of subquotients of a parabolic subgroup $P$ isomorphic to $V\rtimes\Sp_{2n}$. Such subquotients are often the source of subgroups $M\cong\SO_{2n+1}$ of $G$ which are not contained in any Levi subgroup $L$ of $P$. Thus such a subgroup $M$ is non-$G$-completely reducible in the terminology of Serre \cite{Ser05}. It is the presence of non-$G$-completely reducible subgroups which obstructs a recursive classification of the lattice of all reductive subgroups of reductive groups from the list of maximal subgroups.\end{remarks}

\subsection{Pinning.}\label{pinningq}
Recall that a \emph{pinning} of a split reductive $k$-group $G$ is a choice of a split maximal torus $T$, a system of positive roots $\Phi^+\subseteq\Phi$, and for each simple root $a\in\Phi^+$ an isomorphism $x_a:\Ga\to G_a; t \mapsto x_a(t)$ for each root group $G_a$ of $G$. Because of the $T$-equivariant isomorphism $G_a\cong \Lie(G_a)$ this is equivalent to choosing a nonzero element $e_a\in \Lie(G_a)$ for each simple root $a$. The following well-known statements hold:
\begin{itemize}
\item[(C1)] The Chevalley commutator formula: for linearly independent $a,b\in \Phi$ we have \[[x_a(t),x_b(u)]=\prod_{\stackrel{i,j>0}{ia+jb\in \Phi}}x_{ia+jb}(c_{ij}t^iu^j),\] where $c_{ij}\in\Z$ depend only on the Chevalley basis and an ordering of positive roots. 
\item[(C2)] For $t\in \Gm$ and $a\in\Phi$ let $s_a(t):=x_a(t)x_{-a}(-t^{-1})x_a(t)$, and $h_a(t):=s_a(t)s_a(-1)$. Then $h_a$ defines a (coroot) homomorphism $h_a:\Gm\to T$. The images of the $h_a$ as $a$ varies over all simple roots $a$ generate the maximal torus $T$ of $G$. 
\item[(C3)] The $k$-subgroup $N_G(T)$ is generated by the $s_a(t)$ over $t\in \Gm$ and simple roots $a$. The element $s_a(1)$ maps into the Weyl group $W:=N_G(T)(k)/T(k)$ to a reflection $w_a$ in the hyperplane perpendicular to $a^\vee=h_a$; hence $W$ is generated by the images of the $s_a(1)$ over the simple roots $a$.
\item[(C4)] We have $s_a(1)x_b(t)s_a(1)^{-1}=x_{w_a\cdot b}(ct)$ for $c\in\{\pm1\}$ depending only on the Chevalley basis.\label{weylonroot}
\item[(C5)] We have $h_a(t)x_b(u)h_a(t)^{-1}=x_b(t^{\langle b,a^\vee\rangle}u)$.
\item[(C6)] We have $s_a(1)h_a(t)s_a(1)^{-1}=h_a(t^{-1})$, $h_a(t)s_a(u)=s_a(tu)$, and $s_a(t)s_a(v)=h_a(-t/v)$.\end{itemize}
  See \cite[p23 and Lems.~19,~20,~22,~28]{Ste16}; the very last relation in (C6) is a consequence of the other relations which we record here for later use. Note that as the Weyl group is transitive on roots of the same length, we can use (C4) to recover a pinning of $G$ from one function $x_a$ for each root length in $\Phi$, together with the simple reflections $s_i:=s_{a_i}(1)$ where $\{a_1,\dots,a_n\}$ is a set of simple roots.

With this in mind, consider $L\cong\Sp_{2n}$ as a Chevalley group. As per \cite{Bourb82}, we label the simple roots $\{a_1,\dots,a_{n-1},2a_n\}$ according to the Dynkin diagram 
\[\dynkin[edge length=.75cm, labels={a_1,a_2,a_{n-2},a_{n-1},2a_n}] C{},\] 
so that $a_i$ is short for $i<n$ and $2a_n$ is long. Then $L=\langle x_{a_1}(t),x_{2a_n}(t),s_i\rangle$, where  $s_i:=s_{a_i}(1)$ for $1\leq i\leq n-1$ and $s_n:=s_{2 a_n}(1)$.
As in the previous subsection, let $Q$ be the semidirect product of $L=\Sp_{2n}$ with its natural module $U$ and $\pi:Q\to L$ the projection. We have seen that $Q$ contains a subgroup $M=\SO_{2n+1}$ sharing a maximal torus $D$ with $L$. The (implicit) choice of positive roots of $L$ relative to $D$ now implies a presentation of $M$ as a Chevalley group compatible with that of $L$, in a manner that we now spell out. 


Thanks to Proposition \ref{sdpprop},  $\pi$ sends $M$ onto $L$ and has infinitesimal unipotent kernel. So it sends the maximal torus $D$ of $M$ isomorphically onto its image. Hence $\pi$ sends $N_M(D)$ onto $N_L(D)$ (cf.~\cite[Lem.~3.2.1]{CGP15}).

Finally, since the centre $Z$ of $\GL(U)$ acts on $Q$ by centralising $L$ and scaling $U$ linearly, any non-zero scalar multiple of $e_a$ could arise in this way. And since $\pi(N_M(D))=N_L(D)$ acts transitively on the non-zero weight vectors in $U$, by (C4) a choice of $e_{a_n}$ determines each $e_a$, and in turn determines $M$ completely as a subgroup of $Q$. In particular, we get a pinning of $M$, based on the Dynkin diagram
\[\dynkin[edge length=.75cm, labels={a_1,a_2,a_{n-2},a_{n-1},a_n}] B{}.\] 

Now the next result follows and incidentally completes the proof of Proposition \ref{sdpprop}.
\begin{cor}\label{conjcor} The subgroup $M$ is unique in $Q$ up to $Q(k)$-conjugacy.\end{cor}

Note that $\Phi(Q)$ contains three root lengths: from now on we call the $2n$ shorter roots in $\Phi(M)$ \emph{very short roots}; we call the $2n^2-2n$ longer roots in $\Phi(M)$ and shorter roots in $\Phi(L)$ \emph{short roots}; we call the $2n$ longer roots in $\Phi(L)$ \emph{long roots}. [In \cite{CGP15} roots of these lengths are referred to as \emph{multipliable}, \emph{non-divisible and non-multipliable}, and \emph{divisible}, respectively.]

For explicit calculations, we can work in $M$ to calculate commutators of very short and short root elements of $Q$. We can use $L$ to calculate commutators of short and long root elements. Applying the bijection $\pi$ and noting that long root elements from non-opposite root groups commute in $\Sp_{2n}$ we get:

\begin{itemize} \item[(C7)] Let $b$ be a very short root and $2a$ a long root. Then provided $2b\neq -2a$, we have for any $t,u\in k$, $[x_b(t),x_{2a}(u)]=1$; indeed, in that case $[M_b,L_{2a}]=1$, since $2a+b$ is not a weight of $U$.\end{itemize}

\section{Definition of \texorpdfstring{$G$}{G}}\label{defng}
We start by recalling some  concepts from \cite{CGP15}. Let $k$ be a field and $G$ be a pseudo-reductive $k$-group. Let $K/k$ be the minimal field of definition for the geometric unipotent radical $\RR_u(G_{\bar k})$. Let $G^{\mathrm{red}}_K:=G_K/\RR_u(G_K)$ be the maximal reductive quotient of $G_K$. 
The adjunction of Weil restriction and extension of scalars implies that for any finite field extension $k'/k$ the identity map $G_{k'}\to G_{k'}$ corresponds to a $k$-homomorphism 
\begin{equation}\label{canonemb}j_{G,k'}:G\to \R_{k'/k}(G_{k'})\end{equation}
that is a closed immersion.
The composition of $j_{G,K}$ with the Weil restriction of the canonical $K$-quotient map $G_K \to G^{\mathrm{red}}_K$ gives rise to a map \begin{equation}\label{importantmap} i_G:G \to \R_{K/k}(G^{\mathrm{red}}_K).\end{equation}
Importantly, the map $i_G$ is not always a closed immersion. The scheme-theoretic intersection of $\ker i_G$ with any Cartan subgroup $C$ of $G$ is central in $G$ and independent of $C$ by \cite[Prop. 9.4.2(i)]{CGP15}, and one says $G$ is of \emph{minimal type} if this intersection is trivial.  

We recall further that $G$ is called \emph{pseudo-simple} if it is non-commutative and has no non-trivial proper smooth connected normal $k$-subgroups; note that a pseudo-simple $k$-group $G$ is equal to its derived group $\D(G)$ 
\cite[Def.~3.1.1, Lem.~3.1.2]{CGP15}.
We say that $G$ is \emph{absolutely pseudo-simple} if $G_{k_s}$ is pseudo-simple for $k_s$ the separable closure of $k$ in $\bar k$. Say $G$ is \emph{pseudo-split} if it contains a split maximal torus.  

If $G$ is a pseudo-split absolutely pseudo-simple $k$-group of minimal type (so $\Phi$ is irreducible by \cite[Lem. 3.1.4]{CGP15}), then by \cite[Thm. 9.4.7]{CGP15} $\ker i_G$ itelf is non-trivial if and only if $k$ is imperfect, $\Char(k)=2$ and the irreducible root system of $G_{k_s}$ is non-reduced, i.e. a root system of type $BC_n$. Furthermore, \cite[Thm.~2.3.10]{CGP15} gives us  $G_{K}^\mathrm{red}\cong\Sp_{2n}$ and from \cite[Thm.~3.4.6]{CGP15} we see that for any chosen split maximal torus in $G$, there is a Levi subgroup of $G$ isomorphic to $\Sp_{2n}$ containing that torus.

Using the $k$-groups $L$, $M$, $U$, $Q=U\rtimes L$ and $D$ from the previous section, 
we construct pseudo-split pseudo-simple groups of minimal type with non-reduced root systems. 
The basic idea is first to take a Weil restriction $\smash{\QQ:=\R_{E/k}(Q_E)}$ across a suitable purely inseparable field extension $E/k$. This has the effect
of enlarging each of the $2n^2+2n$ root groups of $Q$ to become isomorphic to the additive group of the field $E$---recall that $\Phi(Q)$ contains $2n$
very short roots, $2n^2-2n$ short roots, and $2n$ long roots.
We then generate the desired subgroups from certain vector subgroups of these root groups. More precisely, given any root $a$ in the root system $\Phi:=\Phi(Q,D)$ and root group $\QQ_a=\R_{E/k}(Q_a)$ of $\QQ$, we may take any $k$-vector subspace $S$ of $\QQ_a(k)=Q_a(E)=E$ and find a uniquely corresponding vector subgroup $\underline{S}$ of $\QQ_a$ whose $k$-points coincide with $S$. To decide which parts of which root groups will be required, we need the following data.

\begin{defn}[{cf.~\cite[9.6.8]{CGP15}}]
	\label{Def:data}
	Let $k$ be an imperfect field of characteristic $2$ and let $K/k$ be a non-trivial purely inseparable finite field extension. Fix a pair of non-zero $K^2$-subspaces $V^{(2)}$, $V'\subseteq K$, such that the following conditions hold:
\begin{enumerate}\item $V'$ is a $kK^2$-subspace of $K$;
\item $V^{(2)}\cap V'=0$; and
\item $\dim_{K^2} V^{(2)}< \infty$.\end{enumerate}
Let $K_0:=k\langle V^{(2)}\oplus V'\rangle$ be the subfield of $K$ generated over $k$ by ratios of nonzero elements of $V^{(2)}\oplus V'$. 
Note that since $K_0$ is a field extension of $k$, by definition, and each of $V^{(2)}$ and $V'$ is a non-zero $K^2$-space, $K_0$ contains the subfield $kK^2$ of $K$. 
\end{defn}

\begin{remarks}(i). It turns out that the groups we construct will have $K/k$ as the minimal field of definition over $k$ for the geometric unipotent radicals. In the rank $1$ case we will even have $K=K_0$.

(ii). \label{vfromv2} The notation $V^{(2)}$ is chosen because $V^{(2)}$ is a $K^2$-vector space arising as the image under the Frobenius morphism $q$ of a $K$-vector space $V\subset K^{1/2}$; the map $q$ being the squaring map. Explicitly, to define $V$ in terms of $V^{(2)}$ one defines the abstract $K$-vector space $V:=K\otimes_{\iota,K^2}V^{(2)} \subseteq K\otimes_{\iota,K^2}K$ where $\iota:K^2\to K$ is the square root isomorphism. Since $\smash{K\otimes_{\iota,K^2}K}$ is a field, we have an inclusion of $K$-vector spaces $\smash{V \hookrightarrow K^{1/2}}$ given by $a \otimes b \mapsto a \sqrt{b}$. \end{remarks}

We now define the objects of interest. Fix an integer $n \geq 1$ (which will end up being the rank of our group). 

If $n\geq 2$, then let $(K/k, V',V^{(2)})$ be a triple as in Definition \ref{Def:data}.
If $n=1$, then choose such a triple with one \emph{additional} requirement: in this case, we insist on a choice of $V'$ so that $K_0=K$.

Let $V$ be the (finite-dimensional) $K$-vector space described in Remark \ref{vfromv2}$(ii)$. Let $L$, $M$, $U$ and $Q=U\rtimes L$ be $k$-groups as described in Proposition \ref{sdpprop}.

Let $E$ be the minimal extension of $K$ in $K^{1/2}$ containing $V$, noting that this is a non-zero finite purely inseparable extension (it is generated as an extension by a $K$-basis of $V$, each element of which squares to an element of $K$). Now, let $\QQ:=\R_{E/k}(Q_E)$, and let $\mathsf{M}:= \R_{E/k}(M_E)$, $\mathsf{L}:= \R_{E/k}(L_E)$ and $\mathsf{U}:=\R_{E/k}(U_E)$ be the corresponding subgroups, where $Q = U\rtimes L$ implies $\QQ=\mathsf{U}\rtimes \mathsf{L}$. Abusing notation, we write $\pi$ for the quotient map $\R_{E/k}(\pi):\QQ\to \mathsf{L}$.
Using (\ref{canonemb}) 
we may identify copies of $L$ and $M$ inside $\mathsf{L}$ and $\mathsf{M}$ 
and the common split maximal torus $D$ of $L$ and $M$ identifies with a common split maximal torus of $\mathsf{L}$ and $\mathsf{M}$, which is also a split maximal torus of $\QQ$. The set of roots $\Delta:=\{a_1,\dots,a_n,2a_n\}$ is such that $\{a_1,\ldots,a_n\}$ forms a base of $\Phi(\mathsf{M},D)$ and $\{a_1,\ldots,a_{n-1},2a_n\}$ forms a base of $\Phi(\mathsf{L},D)$. 
For a root $a$, the Weil restrictions of the maps $x_a:(\Ga)_E\to (L_a)_E\subset L_E$ or $x_a:(\Ga)_E\to (M_a)_E\subset M_E$ determine parametrisations of the corresponding root groups $\mathsf{M}_a$ or $\mathsf{L}_a$, giving isomorphisms with $\R_{E/k}(\Ga)$, whose $k$-points are isomorphic to $E$.
Now, define the following subgroups:
\begin{itemize}\item Let $G_{a_n}$ be the $k$-vector subgroup $\underline{V}$ of $\mathsf{M}_{a_n}$, with $\underline{V}(k)=V$.
\item  Let $G_{2a_n}$ be the $k$-vector subgroup $\underline{V'}$ of $\mathsf{L}_{2a_n}$, with $\underline{V'}(k)=V'$ (this makes sense, since $V'$ is chosen to be a $kK^2$-subspace of $K$, and hence is a $k$-subspace in particular).
\item If $n\geq 2$, let $G_{a_1}$ be the $k$-vector subgroup $\underline{K}$ of $\mathsf{L}_{a_1}$ such that $\underline{K}(k)=K$.
\item Finally, let $C:= \R_{K/k}(D_K)$, canonically a subgroup of $\R_{E/k}(D_E)\subset \mathsf{M}\cap\mathsf{L}$. 
\end{itemize}

(For a visual representation when $n=2$, see Figure \ref{BC2} below, taking $V''=K$.)

\begin{defn}\label{Defn:G}
With notation as above, let 
\begin{equation*}
G:= \langle G_{a_1},G_{a_n},G_{2a_n},C,s_i\mid 1\leq i\leq n\rangle \subset \QQ, 
\end{equation*}
where $s_i:=s_{a_i}(1)\in N_L(D)(k)=N_M(D)(k)$ are as defined in \S\ref{pinningq}.
\end{defn}

\begin{lemma}\label{lem:C_G(D)}
(i) The image of $G$ under $\pi$ is a subgroup of the canonical subgroup $\R_{K/k}(L_K)$ of $\mathsf{L}$. 

(ii) We have $C_G(D) = C$. In particular, $C$ is a Cartan subgroup of $G$.

(iii) The group $G$ is connected and smooth.
\end{lemma}
\begin{proof} Part (i) follows by noting that $\pi$ takes each generating subgroup into $\R_{K/k}(L_K)$. (For $t\in E$, an element $x_{a_n}(t)$ maps to $x_{2a_n}(t^2)$, which is in $L(K)$, since $E^2\subseteq K$; and as $\R_{E/k}(M_a)\cong \R_{E/k}(\Ga)$ is a vector group, its $k$-points $M_a(E)$ are dense.)

For part (ii), it follows from \cite[Prop.~A.5.15(3)]{CGP15} that $C_{\mathsf{Q}}(D) = \R_{E/k}(C_{Q_E}(D_E))$.
Since zero is not a weight of $D$ on $U$, we see that $C_{Q_E}(D_E) = C_{L_E}(D_E) = D_E$,
and hence $C_{\mathsf{Q}}(D) = \R_{E/k}(D_E)$.
Therefore, we have $C_G(D) = \R_{E/k}(D_E)\cap G$. But as $\pi$ is an isomorphism on $\R_{E/k}(D_E)$, and by (i), we have $\pi(G)\subseteq \R_{K/k}(L_K)$, and so also $C_G(D)\subseteq \R_{K/k}(L_K)$, hence the assertion.

Part (iii) will follow if we can show that $G$ is generated by smooth connected subgroups. As such, we aim to show that each $s_i$ is contained in a smooth connected subgroup of $G$. We appeal to (C2) and (C4) mainly. First note that if $n\geq 2$ and $i<n$, the simple short root $a_i$ is conjugate to $a_1$ under an element $w\in W$ represented by $s$, which is a product of the $s_i$. Therefore conjugating $x_{a_1}(1)\in G(k)$ by $s$, we may find $x_{a_i}(1)\in G(k)$. Conjugating once more by $s_i$, we also find $x_{-a_i}(1)\in G(k)$. Thus $s_i$ is in the set of $k$-points of the product of three smooth connected groups, namely ${}^sG_{a_1}{}^{s_is}G_{a_i}{}^sG_{a_1}$, as required. Finally, take any $x_{2a_n}(t)$ such that $0\neq t\in V'$. Since $V'$ is a $K^2$-subspace of $K$, we have $t^{-1}\in V'$. Thus $x_{-2a_n}(t^{-1})={}^{s_n}x_{2a_n}(t)\in G(k)$. We now see $s_{a_n}(t)$ in the $k$-points of ${}^sG_{2a_n}{}^{s_is}G_{2a_n}{}^sG_{2a_n}$, and right multiplying the latter by $C$ using (C6), we can in fact see $s_n$ in the product of four smooth connected groups.
\end{proof}


Since the $s_i$ generate $W$ modulo $D$ and and $W$ is transitive on roots of the same length, we can choose a set of pairs $(w,a)$ with $w\in W$ and $a\in\{a_1,a_n,2a_n\}$ so that the collection $w\cdot a$ contains each positive root of $\mathsf{Q}$ exactly once. In fact, as $p=2$, the choices in (C4) vanish and so 
\begin{equation}x_{w\cdot a}(t):={}^{s_w}x_a(t)\label{wdota}\end{equation} 
is well-defined, where $s_w\in N_G(D)(k)$ is some product of the $s_i$ representing $w$.
Let $G_{w\cdot a}$ be the closed subgroup generated by the elements $x_{w\cdot a}(t)$, where $t\in V$, $V'$, or $K$, depending on $a=a_n$, $2a_n$, or $a_1$ respectively (the last only when $n\geq 2$).

For convenience of exposition, order the positive roots as $\{b_1,\dots, b_{2n},\dots, b_{m}\}$, where $m:=n^2+n$ so that the roots $b_1,b_3,\dots,b_{2n-1}$ are very short and in height order, and for $1\leq i\leq n$, we have $b_{2i}=2b_{2i-1}$ is a long root. Thus $b_i$ is short for $i>2n$.
Now consider the multiplication map
\begin{equation}\label{eq:m} \mu:\prod_{i=1}^{m} G_{b_i}\to G.\end{equation}

\begin{lemma}\label{scru}
The map $\mu$ as just defined is injective on $k$-points. Let $\UU$ be the image of the $k$-points of the domain. Then $\UU$ is a subgroup of $\mathsf Q(k)$. 
\end{lemma}
\begin{proof}
The claim about injectivity will obviously follow if the post-composition with $\pi:\mathsf Q\to \mathsf{L}$ is injective.
Since for a very short root $a$, $\pi(x_a(t))=x_{2a}(t^2)$, we have on $k$-points, 
\[\pi\circ \mu:\left(x_{b_1}(c_1),\dots,x_{b_{m}}(c_{m})\right)\mapsto \prod_{i=1}^n x_{b_{2i}}\left(c_{2i-1}^2+c_{2i}\right)\prod_{i=2n+1}^{m}x_{b_i}(c_i)\in\Sp_{2n}(K).
\]
As the elements in the image of $\pi\circ \mu$ are entirely supported on positive roots, the image lies in $B(K)$ for $B$ a Borel subgroup of $(\Sp_{2n})_{K}$; and so does the subgroup they generate. Since the unipotent radical $\RR_{u,K}(B)$ of $B$ is isomorphic as a $K$-scheme to $\mathbb{A}^{n^2}$ via a direct product of its root groups,
two elements $(x_{b_1}(c_1),\dots,x_{b_{m}}(c_{m}))$ and $(x_{b_1}(d_1),\dots,x_{b_{m}}(d_{m}))$
have the same image in $\Sp_{2n}(K)$ if and only if $c_{2i-1}^2+c_{2i} = d_{2i-1}^2+d_{2i}$ for $1
\leq i\leq n$ and $c_i=d_i$ for $i>2n$.
Moreover, for each very short root $a_{2i-1}$ we have $c_{2i-1}$ and $d_{2i-1}$ naturally identified as elements of $V\subseteq E$; and for each long root, we have similarly $c_{2i},d_{2i}\in V'\subseteq K\subseteq E$, for $1\leq i\leq n$.
But $c_{2i-1}^2, d_{2i-1}^2\in V^{(2)}$ and $V^{(2)}\cap V'=0$, so we deduce that 
$c_{2i-1}^2+c_{2i} = d_{2i-1}^2+d_{2i}$ if and only if $c_{2i-1} = d_{2i-1}$ and $c_{2i} = d_{2i}$,
and this completes the proof of injectivity.

We now show $\UU$ is a subgroup of $\QQ(k)$. 

Let $\mathsf{x},\mathsf{y}\in \UU$. Consider $\mathsf{x}\mathsf{y}$ as a word in root elements of $\mathsf{Q}(k)$. We wish to use the commutator formula (C1) to reorder these root elements in this expression so that the element $\mathsf{x}\mathsf{y}$ is visibly equal to a member of $\UU$ by virtue of the definition of $\UU$. To wit, let $b$ be a very short root. By (C7), if $a\neq b$ is another positive root then $x_a(t)$ commutes with $x_b(u)$ unless $a$ is short and $a+b$ is another very short root. In the latter case, calculating in $\mathsf{M}(k)$ with (C1) we have \[x_a(t)x_b(u)=x_b(u)x_{a+b}(tu)x_a(t)x_{2b+a}(tu^2).\]
(One needs to check in this case that the relevant $c_{ij}=1$ in (C1), see for example \cite[Prop. 12.3.3 proof]{Car89}.)
As $t\in K$ and $u\in V$, we have $tu\in V$, so the right-hand side above is indeed an element of $\UU$. Using (C1) in $\mathsf{L}(k)=L(E)$ we may similarly reorder other pairs of root elements. After a finite number of commutations, one has $\mathsf{x}\mathsf{y}$ re-written as a product of root groups in the order specified just before (\ref{eq:m}). The above calculations show that the coefficients of each root element in that expression are already in the image of the map $\pi\circ \mu$ and so $\UU$ is a subgroup.
\end{proof}

Let $\GG$ be the (abstract) subgroup of $\QQ(k)$ generated by the $k$-points of the generators of $G$; explicitly,  \[\GG:=\langle G_{a_1}(k), G_{a_n}(k), G_{2a_n}(k),C(k),s_i\rangle.\] Our aim is to prove $\GG=G(k)$; in particular, that a root group of $G$ is isomorphic to one of $G_{a_1}$, $G_{a_n}$ or $G_{2a_n}$ depending on the length of the root. To that end we construct a $(B,N)$-pair (or Tits system) for $\GG$ in the sense of \cite[Sec.\ 6.2.6]{AB08}, \cite[Sec.\ 2.1]{Car93}. Standard results imply that $\GG$ has an abstract Bruhat decomposition, which we then compare to the algebraic Bruhat decomposition of $G$.

Recall that a pair of subgroups $(\mathcal{B},\mathcal{N})$ of an abstract group $\mathcal{G}$ form a $(B,N)$-pair if the following axioms are satisfied
\begin{enumerate}\item[(BN1)] $\mathcal{G}$ is generated by $\mathcal{B}$ and $\mathcal{N}$;
\item[(BN2)] $\mathcal{B}\cap \mathcal{N}$ is a normal subgroup of $\mathcal{N}$;
\item[(BN3)] the group $\mathcal{W}=\mathcal{N}/\mathcal{B\cap N}$ is generated by a set of elements $w_i, i\in I$ such that $w_i^2=1$;
\item[(BN4)] if $s_i\in\mathcal{N}$ maps to $w_i$ under the quotient map $\mathcal{N}\to \mathcal{W}$, and if $s$ is any element of $\mathcal{N}$, then 
\[\mathcal{B}s_i\mathcal{B}.\mathcal{B}s\mathcal{B}\subseteq \mathcal{B}s_is\mathcal{B}\cup \mathcal{B}s\mathcal{B},\]
\[\text{i.e. } s_i \mathcal{B}s\subseteq \mathcal{B}s_i s \mathcal{B}\cup \mathcal{B}s \mathcal{B};\]
\item[(BN5)] if  $s_i$ is as above, then $s_i\mathcal{B}s_i\neq \mathcal{B}$.\end{enumerate}

\begin{lemma}\label{bnpairlem}
Let $\BB$ be the subgroup of $G(k)$ generated by $\CC:=C(k)$ and $\UU$,
and let $\NN:=N_G(D)(k)$.
\begin{itemize}
\item[(i)] $\BB$ is the semidirect product of $\CC$ and $\UU$.
\item[(ii)] The subgroups $\BB$ and $\NN$ of $\GG$ form a $(B,N)$-pair for $\GG$ with Weyl group $W:=\NN/(\BB\cap \NN)$ that is generated by the images of the $s_i$ in $W$ with $1\leq i\leq n$.
\item[(iii)] The $(B,N)$-pair in (ii) is saturated; that is,
$$
\bigcap_{w\in W} s_w\BB s_w^{-1} = \CC = \BB\cap\NN.
$$
\item[(iv)] Let
$$
Z := \bigcap_{g\in \GG} g\BB g^{-1}.
$$
Then $Z=1$.
\item[(v)] $\GG$ is the disjoint union of the double cosets $\BB s_w\BB$ with $w\in W$.
\item[(vi)] Every element $x\in\GG$ has a unique expression $u_w s_w h u$, where $u_w$ is a product of positive root elements sent negative by $w$, $h\in \CC$, $u\in \UU$.
\item[(vii)] $\ker\pi|_{\GG}=1$.
\end{itemize}
\end{lemma}

\begin{proof}
(i). We recall $\CC=D(K)$. As $\UU$ is unipotent, and $D$ is multiplicative, $\CC\cap \UU=1$. Since $\CC$ is generated by the elements $h_{a}(u)$, with $a\in\Delta$ and $u\in K^\times$, we see that $\CC$ normalises all the (closed) root groups of $\QQ$.  Then the assertion about $\BB$ follows from the relations $x_b(c)h_a(u)=h_{a}(u)x_b(u^{\langle b,a^\vee\rangle}c)$. (The interesting case is where $b$ is long and $c\in V'$; then $\langle b,a^\vee\rangle$ is always a multiple of $2$, so that $u^{\langle b,a^\vee\rangle}\in K^2$ and its product with $c$ remains in $V'$.) 

(ii). We first need to verify that $\NN\subseteq \GG$. To see this, note that since $\GG$ contains all the elements $s_i$ and $\CC$, $\pi(\NN)$ is
contained in $\pi(\GG)$. But since the weights of $D$ on $\Lie(\ker\pi)$ are non-zero, $\ker\pi\cap\NN=0$ and so $\pi$ is an isomorphism on restriction to $N_G(D)$; thus $\NN\subseteq \GG$.
We now establish each of the conditions BN1--BN5 holds.

\begin{enumerate}\item[(BN1)] Note that the elements of $\CC$ and the $s_i$ are members of $\NN$, and the subgroups $G_{a_1}(k)$, $G_{a_n}(k)$, $G_{2a_n}(k)$, are all contained in $\BB$; therefore, this follows by definition of $\GG$.
\item[(BN2)] We have seen that $\pi$ is an isomorphism on $\CC$, $\UU$ and $\NN$. 
Since $\pi(\CC)$ lies inside the maximal torus and $\pi(\UU)$
lies inside the product of the positive root groups in $\Sp_{2n}(K)$, we have $\pi(\CC)\cap \pi(\UU) = 1$.
Therefore, $\pi$ is an isomorphism on $\BB$ and $\BB\cap \NN$, using part (i).
From the Bruhat decomposition in $\Sp_{2n}(K)$, we have $\pi(\BB)\cap \pi(\NN)=\pi(\CC)=\CC$ so we are done.
\item[(BN3)] As $C_G(D)=C$, by Lemma \ref{lem:C_G(D)}, and $\NN$ contains all the elements $s_i$, we have that $\NN/\CC\cong N_L(D)/D \cong W$, hence is generated by the images of the $s_i$, which are involutions in $W$.
\item[(BN4)] For $i\leq n-1$, let $\Phi_i=\Phi^+\setminus\{a_i\}$ and $\Phi_n=\Phi^+\setminus\{a_n,a_{2n}\}.$ Then define $\UU_i$ to be the subgroup of $\UU$ generated by $G_b(k)$ for $b\in \Phi_i$. It is easy to see from the commutator formulas that $\UU_i$ is the image of the subgroup $\prod_{b\in\Phi_i} G_b(k)$ under the multiplication map $\mu$ from \eqref{eq:m}. Further, by \eqref{wdota} we see that $\UU_i$ is normalised by $s_i$. Now let $\XX_i:=G_{a_i}(k)$ for $1\leq i\leq n-1$ and $\XX_n:=G_{a_n}(k)G_{a_{2n}}(k)$. Then $\UU$ factorises as $\XX_i\UU_i$ for all $i$, using the same argument as in the proof of Lemma \ref{scru}. 

If $b$ is short or long, a standard calculation in a subgroup $\SL_2(E)$ of the Levi subgroup of $\mathsf{L}$ with roots $\pm b$, gives \[x_{-b}(t)=x_{b}(t^{-1})s_bh_{b}(t)x_{b}(t^{-1}),
\] 
where \[h_{b}(t)=x_{b}(t)x_{-b}(t^{-1})x_{b}(t)x_{b}(1)x_{-b}(1)x_{b}(1).\]
So every non-identity element of $G_{-b}(k)$ lies in $\BB s_b\BB$. 
If $b$ is very short then $G_{\pm b}$ lies in $\mathsf{M}_{\pm b}:=\langle \mathsf{M}_b,\mathsf{M}_{-b}\rangle\cong \R_{E/k}(\PGL_2)$, which is a subgroup of $\mathsf{M}$. We have $\pi:\mathsf{M}_{\pm b}\to \mathsf{L}_{\pm 2b}\cong \R_{E/k}(\SL_2)$ is injective on $k$-points, so the same calculation shows that every non-identity element of $G_{-b}(k)$ is contained in $\BB s_b \BB$ as required. Now as $\BB=\XX_i\UU_i\CC$, we get
\begin{equation} {}^{s_i}\BB\subseteq \BB\cup \BB s_i \BB.\label{negroot}\end{equation}

To establish (BN4) we now reproduce the proof one uses to investigate the Chevalley groups, as in \cite[Prop.~8.1.5]{Car89}. Suppose a product $s$ of the $s_i$ represents $w$ in the Weyl group $W$. One takes the cases $w(a_i)\in\Phi^+$ and $w(a_i)\in\Phi^-$ separately. In the first case, we calculate\begin{align*}s \BB s_i&=s \XX_i\UU_i\CC s_i\\
&= {}^s\XX_{i} s s_i \UU_i\CC \\
&\subseteq \BB s s_i \BB,\end{align*}
since ${}^s\XX_i$ is contained in a positive root group or a product of two positive root groups if $i=n$. If $w(a_i)\in \Phi^-$, then the image $w'$ of $ss_i$ has $w'(a_i)\in\Phi^+$. Then
\begin{align*}s \BB s_i&\subseteq ss_i(\BB\cup \BB s_i\BB)\qquad\text{(by \eqref{negroot})}\\
&\subseteq \BB ss_i\BB\cup \BB (ss_i \BB s_i) \BB \quad\text{(by the previous case)}\\
&\subseteq  \BB ss_i \BB \cup  \BB ( \BB ss_is_i \BB ) \BB \quad\text{(by the previous case)}\\
&\subseteq  \BB ss_i \BB \cup  \BB s \BB .\end{align*}
Finally one gets (BN4) by observing that if $g\in s_i \BB s$ then $g^{-1}\in s^{-1}  \BB s_i$. So we deduce $g^{-1}\in  \BB s^{-1}s_i \BB \cup  \BB s_i  \BB $ and $g\in  \BB s_is \BB \cup  \BB s_i  \BB $.

\item[(BN5)] For any element $s_w\in \NN$ mapping to $w\in W$, if $w\neq 1$ then there is a positive root $a$ of $\mathsf{Q}$ sent negative by the action of $w$. Therefore conjugation by $s_w$ sends $x_a(t)$ with $t\neq 0$ outside $\BB$.\end{enumerate}

(iii). The containment $\supseteq$ is clear as $\NN$ normalises $D$ (and $C=C_G(D)$). For the other containment, since $\pi$ is injective on $\BB$ it is injective on any conjugate. The result then follows from the analogous result for $L(K)$ \cite[Sec.~1.10]{Car93}.

(iv). Part (iii) implies that $Z\subseteq \CC$, and hence $Z$ normalises $\UU$. 
Since $\UU$ clearly normalizes $Z$, and $\CC\cap \UU = 1$,
we see that $Z$ in fact centralizes $\UU$. Now using formula (C5), we see that every positive root vanishes on $Z$. Indeed, taking a look at the root system of $\QQ$, one discovers that for each root $a$ there is some root $b$ such that ${}^{h_b(t)}x_a(u)=x_a(tu)$, and hence $Z$ is contained in the centre of $\QQ$, which is trivial.
%

(v). This now follows from \cite[Prop.~8.2.2, Prop.~8.2.3]{Car89}.

(vi). The proof of this follows as in \cite[Thm.~8.4.3, Cor.~8.4.4]{Car89}. 

(vii). For $x\in \ker\pi|_\GG$ write $x=b_1s_wb_2$. Since $\pi(b_1)\pi(s_w)\pi(b_2)=1$, we may use the Bruhat decomposition in $\Sp_{2n}(K)$ to see $w=1$; but then $x\in \BB$ and $1$ is the only preimage of $1$ under~$\pi$.
\end{proof}


Let $\lambda:\Gm\to D$ be a cocharacter which is positive on positive roots; i.e. for which $\langle\lambda,\alpha\rangle>0$ for all $\alpha\in\Phi^+\subset \Phi=\Phi(Q)$.

As in \cite[\S2.1]{CGP15}, define subgroup schemes of $\mathsf{Q}$ on $A$-points via 
\[P_{\mathsf{Q}}(\lambda)(A)=\{g\in \mathsf{Q}(A)\mid \lim_{t\to 0}\lambda(t)g(t)\lambda(t)^{-1}\text{ exists}\}\]
\[U_{\mathsf{Q}}(\lambda)(A)=\{g\in \mathsf{Q}(A)\mid \lim_{t\to 0}\lambda(t)g(t)\lambda(t)^{-1}=1\}.\]

We have $P_{\mathsf{Q}}(\lambda)=Z_{\mathsf{Q}}(\lambda)U_{\mathsf{Q}}(\lambda)$ by \cite[Prop.~2.1.8(2)]{CGP15}. The first factor is the Cartan subgroup $\R_{E/k}(D_E)$ and the second is the unipotent radical, $\RR_u(P_{\mathsf{Q}}(\lambda))=\prod_{a\in\Phi^+}\mathsf{Q}_a$. On the other hand, $U_{\mathsf{Q}}(-\lambda)=\prod_{a\in\Phi^-}\mathsf{Q}_a$. Now \cite[Prop.~2.1.2]{CGP15} tells us that the multiplication map $\mu:U_{\mathsf{Q}}(-\lambda)\times P_{\mathsf{Q}}(\lambda)\to G$ is an open immersion. Let $s_{w_0}$ be a word in the $s_i$ representing the longest word $w_0$ in the Weyl group. Then as $w_0$ induces the inversion map on $D$, we see that the map \[U_{\mathsf{Q}}(-\lambda)\times P_{\mathsf{Q}}(\lambda)\to G;\ (u,p)\mapsto s_{w_0}up = s_{w_0} u s_{w_0}^{-1} s_{w_0} p,\] also induces an open immersion \[\mu':U_{\mathsf{Q}}(\lambda)\times P_{\mathsf{Q}}(\lambda) \to \Omega:=U_{\mathsf{Q}}(\lambda)s_{w_0}P_{\mathsf{Q}}(\lambda).\] Now, the positive root groups of $\mathsf{Q}$ can be picked out as in \cite[Lem.~2.3.3]{CGP15}. Take $D_a=(\ker a)_\red^0$, some $\lambda_a:\Gm\to D$ with $\langle a,\lambda_a\rangle>0$ and then $\mathsf{Q}_{(a)}=U_{Z_\QQ(D_a)}(\lambda_a)$, where if $a$ is short, $\QQ_{(a)}=\QQ_a$ and if $a$ is very short, $\QQ_{(a)}=\QQ_a\times\QQ_{2a}$---the $k$-points of the two factors commute and are dense in each. With the above in place, we are now ready to prove:

\begin{lemma}\label{rootgroupswork} For each root $a$, if we let $U_a$ denote the $a$-root group of $G$, then we have $U_a = G_a$. 
In particular, $U_a$ is isomorphic to $G_{a_1}\cong \underline{K}$, $G_{a_n}\cong \underline{V}$ or $G_{a_{2n}}\cong \underline{V'}$ according to whether $a$ is short, very short or long, respectively.

Moreover, $G(k)=\GG$.\end{lemma}
\begin{proof} Observe the closure $\overline{\GG}$ is a subgroup of $G$ containing all the generating groups for $G$, and therefore is equal to it. 

Since $\Omega$ meets $\GG$ non-trivially, it meets $G$ non-trivially, and hence $\Omega_G:=G\cap \Omega$ is non-empty and open in $G$. Now, the set of points of $\GG$ contained inside $\Omega_G(k)$ is, by Lemma \ref{bnpairlem}(v), precisely $\Omega_\GG:=\BB s_{w_0}\BB$: for if, a fortiori, $\Omega(k)$ hits another $(\BB,\BB)$-double coset in $\GG$, then 
applying $\pi$ and using the Bruhat decomposition in $L(E)$ we see that such a double coset must be determined by the same element of the Weyl group, namely $w_0$, hence is equal to $\BB s_{w_0}\BB$.

The complement $G\setminus\Omega_G$ is closed and since it contains the $k$-points $\GG\setminus\Omega_\GG$, the closure of $\GG\setminus\Omega_\GG$ does not meet $\Omega_G$. Thus $\Omega_\GG$ is dense in $\Omega_G$. Therefore the preimage $(\mu')^{-1}(\BB s_{w_0}\BB)$ is dense in $(\mu')^{-1}(\Omega_G)=(U_{\mathsf{Q}}(\lambda)\cap G) \times (P_{\mathsf{Q}}(\lambda)\cap G)$. But $U_{\mathsf{Q}}(\lambda)\cap G=U_G(\lambda)$ is isomorphic as a $k$-scheme to the product $\prod_{a\in\Phi^+} U_a$ of the root groups $U_a$ of $G$, and $P_{\mathsf{Q}}(\lambda)\cap G=Z_G(\lambda)U_{G}(\lambda)$ for $Z_G(\lambda)$ a Cartan subgroup of $G$, which according to Lemma \ref{lem:C_G(D)} is $C$. The preimage $(\mu')^{-1}(\Omega_\GG)$ is $\prod_{a\in\Phi^+} G_a(k)\times \BB$ because this set surjects onto $\Omega_\GG$ and $\mu'$ is injective.
Since $\Omega_\GG$ is dense in $\Omega_G$, the projection of this preimage to $U_\QQ(\lambda)=\prod_{a\in\Phi^+}\QQ_{a}$ can be used to determine the root groups of $G$. For $a$ short, $(a)=a$ and we have $U_a=U_{(a)}=\QQ_a\cap G$. Now $G_a(k)=K$ closes up to $G_a\cong\underline{K}$; thus a short root group of $G$ is isomorphic to $G_{a_1}$. On the other hand, we hit the factor $\QQ_{(a_n)}=\QQ_{a_n}(k)\times \QQ_{2a_n}(k)$ in the set $\{(x_{a_n}(t),x_{2a_n}(t^2)x_{2a_n}(u))\mid t\in V, u\in V'\}$. Under the isomorphism $\QQ_{a_n}\times \QQ_{2a_n}\cong (\mathbb A^2)_E$, we may write the latter as $\{(t,t^2+u)\mid t\in V, u\in V'\}\cong V\times V'$. The closure of this set is evidently the closed subgroup $\underline{V}\times \underline{V'}$ embedded via $(t,u)\mapsto (t,t^2+u)$. In $\QQ_{a_n}\times \QQ_{2a_n}$ this corresponds to the closed subgroup $G_{a_n}G_{2a_n}$. Thus a very short root group of $G$ is isomorphic to $G_{a_n}$ and a long root group of $G$ is isomorphic to $G_{2a_n}$, as required.

Now we see that $\GG = G(k)$ by using the Bruhat decomposition for $G$ with respect to $B=P_G(\lambda)$ \cite[Theorem~C.2.8]{CGP15}:
since the root groups of $G$ are as expected, we see that $B(k) = \BB$ and so the result follows.
\end{proof}

The above results furnish us with a method to see that $G$ is also pseudo-reductive. (In \cite[\S9]{CGP15} the most common  tool used is \cite[Lem.~7.2.4]{CGP15} which says that a subgroup of a pseudo-reductive group containing a Levi is also pseudo-reductive.) 

\begin{prop}The group $G$ is pseudo-reductive of minimal type; $\D(G)$ is absolutely pseudo-simple of minimal type.\label{ispseudo}
\end{prop}
\begin{proof}
For the first part, since $G$ is pseudo-reductive if and only if $G_{k_s}$ is, it suffices to show that the $k_s$-points of $\RR_{u,k_s}(G_{k_s})$ are trivial since they are dense. 
Now by \cite[Prop.~2.47]{milne17}, $G_{k_s}$ is generated by the base change to $k_s$ of each of $G_{a_1}$, $G_{a_n}$, $G_{2a_n}$ and $C$ together with $s_1,\dots s_n$. Let $K':=K\otimes_k k_s$ and $E':=E\otimes_k k_s$, which are both fields, since $E/k$ and $K/k$ are purely inseparable. 
It is easy to see that $G_{k_s}$ is a $k_s$-group satisfying Definition \ref{Defn:G}; we therefore assume $k=k_s$ and proceed.

Since $\RR:=\RR_{u,k}(G)(k)$ is a normal subgroup of $\GG=G(k)$, we conclude from \cite[Lem.~6.61]{AB08} that either $\RR\subseteq Z=\bigcap_{g\in \GG}g^{-1}\BB g$ or $\RR \BB=\GG$. The latter is impossible since $\GG$ is not solvable; so by Lemma \ref{bnpairlem}(v), we deduce $\RR=1$ as required.

This proves the first part. 
Since $G$ is pseudo-reductive, $\D(G)$ is perfect by \cite[Prop.~1.2.6]{CGP15}, and hence \cite[Lem.~3.1.2]{CGP15} implies that $\D(G)$ is absolutely pseudo-simple. 
For the final point, that $\D(G)$ has minimal type, recall that the quotient map $\pi:\mathsf Q\to \mathsf{L}$ restricts to a closed immersion on $C$; in particular, the intersection of $\ker \pi$ with $C$ is trivial. 
Hence $G$ is of minimal type by Lemma \ref{lem:C_G(D)}(ii). It is then immediate from \cite[Prop.~9.4.5]{CGP15} that $\D(G)$ is of minimal type.
\end{proof}


Given an arbitrary absolutely pseudo-simple pseudo-split $k$-group $G$ of minimal type, one wishes to characterise it up to isomorphism by the data of the form given in 
Definition \ref{Def:data}. In particular, one wants to calculate the minimal field of definition of its unipotent radical. Owing to the maps (\ref{canonemb}) and (\ref{importantmap}), this is tantamount to finding a Levi subgroup. For the groups we have constructed in Definition \ref{Defn:G}, this process is rather easier if one assumes $1\in V'$.

To that end, we give our version of \cite[Prop.~9.7.10]{CGP15} which discusses a natural automorphism $i_\lambda$ of $\Sp_{2n}$ scaling the long root groups by $\lambda$ and centralising $C$ and the simple short root groups; \textit{loc.~cit.} extends $i_\lambda$ to an isomorphism $G\to i_\lambda(G)$ by passing through the group law construction. We are able to take a more undemanding approach by viewing $i_\lambda$ as conjugation inside $\mathsf{Q}$.

Consider $\Sp_{2n}(E)$ as elements of $\GL_{2n}(E)$ preserving the standard symplectic form $J=\begin{pmatrix}0& I\\-I& 0\end{pmatrix}$, i.e. $g^TJg=J$. We identify $D$ with a diagonal torus. For $\lambda\in E^\times$ one checks $h(\lambda):=\begin{pmatrix}\lambda I& 0\\0 &\lambda^{-1}I\end{pmatrix}\in \Sp_{2n}(E)$, defining a cocharacter $h:\Gm\to \mathsf{L}$. Conjugation by $h(\lambda)$ gives
\[h(\lambda)\begin{pmatrix}a & b\\c &d\end{pmatrix} h(\lambda)^{-1}=\begin{pmatrix}a & \lambda^2 b\\\lambda^{-2}c &d\end{pmatrix},\]
where $a$, $b$, $c$ and $d$ are $n\times n$ block matrices. The centraliser $\mathsf{L_0}$ of the image $D_0$ of $h$ is the stabiliser of the two totally isotropic subspaces $\langle e_1,\dots,e_n\rangle$ and $\langle e_{n+1},\dots, e_{2n}\rangle$. As such $\mathsf{L_0}$ is a Levi subgroup of $\Sp_{2n}$ of type $A_{n-1}$, with central torus $D_0$; and indeed $D_0$ centralises the short simple root groups of $\mathsf{L}$ and their negatives. On the other hand, since a long root group is contained in an $\R_{E/k}(\Sp_{2})\cong\R_{E/k}(\SL_2)$-subgroup of $\mathsf{L}$ stabilising a non-degenerate $2$-dimensional subspace of $\overline{X}$, we can take $a=b=d=1$, $c=0$ in the matrix above with $n=1$, and observe that $h(\lambda)$ scales a positive long root group by $\lambda^2$; similarly, $h(\lambda)$ scales a negative long root group by $\lambda^{-2}$. As the set of weights of $D$ on the normal subgroup $\mathsf{U}$ of $\mathsf{Q}$ is the set of very short roots, and every very short root is half a long root, we see that conjugation by $h(\lambda)$ scales a positive very short root group by $\lambda$ and its negative by $\lambda^{-1}$.
\begin{lemma}\label{lem:1inv'} Suppose $G$ is constructed from data $(K/k,V^{(2)},V')$ as in Definition \ref{Def:data}; and $\lambda\in E^\times$. Then conjugation by $h(\lambda)$ induces an automorphism of $\mathsf{Q}$ in which $G$ is sent to the isomorphic subgroup $G'$ of $\mathsf{Q}$ constructed via the data $(K/k,\lambda^2 V^{(2)},\lambda^2 V')$.

Thus $G$ is $k$-isomorphic to a group constructed from data in which $1\in V'$.\label{1inV'}\end{lemma}
\begin{proof}We consider the action of conjugation by $h(\lambda)$ on the generators in Definition \ref{Def:data}. We have already checked the action of $h(\lambda)$ on the simple root groups and $C$. Thus for $i<n$, we have $s_i:=x_{a_i}(1)x_{-a_i}(1)x_{a_i}(1)$ is centralised by $h(\lambda)$. If $i=n$, then  conjugation by $h(\lambda)$ sends $s_i=x_{2a_n}(1)x_{-2a_n}(1)x_{2a_n}(1)$ to $x_{2a_n}(\lambda^2)x_{-2a_n}(\lambda^{-2})x_{2a_n}(\lambda^2)=h_{2a_n}(\lambda^2)\cdot s_i$. Since $h_{2a_n}(\lambda^2)\in C(k)$, ${}^{h(\lambda)}s_i\in G'$ for all $i$, whence the first claim.

Now let $0\neq v\in V'\subseteq K$. 
It does no harm in our constructions to enlarge $E$ to a bigger finite intermediate extension in $K^{1/2}/K$, so we may assume $\sqrt{v}\in E$. Conjugating by $h(\sqrt{v})$ takes $G$ to a group $G'$ defined by the data $(K/k,vV^{(2)},vV')$. Since $V'$ is a $K^2$-subspace it follows that $1\in vV'$.\end{proof}

\begin{lemma}\label{igmap} If $1\in V'$, then the derived group $\D(G)$ contains the canonical $k$-subgroup $L\subseteq \R_{K/k}(L_K)$. Indeed $L$ is a Levi subgroup of both $G$ and $\D(G)$.

In any case, the minimal field of definition of $\RR_u(G_{\bar k})$ is $K$ and the restriction $\pi|_G$ of the quotient map $\pi:\mathsf Q\to \mathsf{L}$ identifies with the map $i_G:G\to \R_{K/k}(\Sp_{2n})$.
\end{lemma}
\begin{proof}  Since $[h_a(c),x_b(t)]=x_b(c^{\langle b,a\rangle}t+t)$, we can reproduce any root element  as a commutator; thus $\D(G)$ contains all the root groups of $G$. As $1\in V'$, we have $k\subseteq V'$ and so $x_{2b}(t)\in G(k)$ for $t\in k$ for any long root $2b$. Besides, $x_a(t)\in G(k)$ for all short roots $a$ and $t\in k$. As these elements are a set of generators of $L$, we have $L\subset \D(G)$.

For the second part, by Lemma \ref{1inV'}, we may assume $1\in V'$. Now, we argue as in Step 3 of the proof of \cite[Thm.~9.8.1]{CGP15}. Let \[q_{L_K}: \R_{K/k}(L_K)_K \to L_K\] be the $K$-morphism associated to the identity map $\R_{K/k}(L_K) \to \R_{K/k}(L_K)$ under the adjunction property of Weil restriction and extension of scalars. Since the composition \[L_K\longrightarrow G_K\stackrel{(\pi|_G)_K}{\longrightarrow} \R_{K/k}(L_K)_K\stackrel{q_{L_{\!K}}}\longrightarrow L_K\] is an isomorphism, we have $G_K\cong \ker\psi\rtimes L_K$ where $\psi=q_{L_K}\circ(\pi|_G)_K$. Then as $G_K$ is smooth and connected, so is $\ker\psi$. Lastly, as $\psi$ is a surjective map between two groups of rank $n$, $\ker\psi$ contains no $K$-tori. Hence $\ker\psi$ is unipotent and so the minimal field of definition of $\RR_{u}(G_{\bar k})$ is contained in $K$. Since $G$ contains the subgroup $C\cong \R_{K/k}(\Gm^n)$, and $C_K\cap \ker\psi$ is the geometric unipotent radical of $C_K$, the minimal field of definition of $\RR_{u}(G_{\bar k})$ must contain $K$. 

Moreover, the map $\psi$ must be the quotient of $G_K$ by its unipotent radical and by naturality of Weil restriction, we see $i_G:G\to\R_{K/k}(L_K)$ identifies with the restriction of $\pi$ to $G$.
\end{proof}

\begin{remark} In fact, Step 5 of the proof of \cite[Thm.~9.8.1]{CGP15} shows that $K$ is also the minimal field of definition of $\RR_u(\D(G)_{\bar k})$. Evidently it is a subfield of $K$. To show the other inclusion, one deals with the case $n=1$ separately and assuming $n\geq 2$, one shows that for a short root $b$, the subgroup with $k$-points $\langle x_b(t),x_{-b}(t)\mid t\in K\rangle$ is isomorphic to $\R_{K/k}(\SL_2)$ and contains a Cartan subgroup isomorphic to $\R_{K/k}(\Gm)$, for which the minimal field of definition of the geometric unipotent radical is $K$.
 \end{remark}

In case the rank $n$ of $G$ is $2$, there are more possibilities for the isomorphism class of $G$, which we describe now. For a visual representation, see Figure \ref{BC2} below. We need to fix some additional notation: 
Given $(K/k,V^{(2)},V')$ as in Definition \ref{Def:data}, take $V''$ a proper $K_0$-subspace of $K$ such that $k\langle V''\rangle=K$ where $k\langle V''\rangle$ denotes the subfield of $K$ generated by $k$ and the ratios of nonzero elements of $V''$. 
For the simple roots $a_1,a_2,2a_2$, let $G_{a_2}(k)=V$ and $G_{2a_2}(k)=V'$ as before, but set $G_{a_1}=\underline{V''}$ the vector subgroup of the $a_1$-root subgroup of $\mathsf{L}$ whose $k$-points are $V''\subseteq K\subseteq E=\mathsf{L}_{a_1}(k)$. 
Finally, take $\mathsf{s}_1=s_{a_1}(v_1)$ and $\mathsf{s}_2=s_{2a_2}(v_2)$, for some $0\neq v_1\in V''$ and $0\neq v_2\in V'$.

\begin{figure}[b]\begin{mdframed}
\centering
\captionsetup{font=small,labelfont=bf,
   justification=justified,
   format=plain}
   
    \begin{minipage}[b]{0.38\textwidth}
    \centering
\usetikzlibrary{arrows,positioning}
\begin{tikzpicture} [%
    blue/.style = {circle,fill=blue,text=black,inner sep=3pt},
    red/.style = {circle,fill=red,text=black,inner sep=3pt},
   green/.style = {circle,fill=green,text=black,inner sep=3pt}, 
    line/.style={->,shorten >=0.4cm,shorten <=0.4cm},thick, scale=0.6, every node/.style={transform shape}]

\draw [brown, thin,fill=none,-] (-0.2,-0.3) node [above left, black] {} to [out=-120,in=30](-2.1,-2.1) node[left, black] {$(V'')^*_{K/k} \times (V_0)^*_{K_0/k}$};
\draw [black, thick,fill=none,->_______________________] (0,0) node [above left, black] {} -- (-2,0) node[above left, black] {$V\hspace{0.6mm}$};
\draw [black, thick,fill=none,->______________________] (-2,0) node [above left, black] {} -- (-4,0) node[above left, black] {$V''\hspace{-0.2mm}$};
\draw [black, thick,fill=none,->_______________________] (0,0) node [above left, black] {} -- (2,0) node[above right, black] {$\hspace{0.4mm}V$};
\draw [black, thick,fill=none,->______________________] (2,0) node [above left, black] {} -- (4,0) node[above right, black] {$V''$};
\draw [white, thick,fill=none,->__________] (0,0) node [above left, black] {} -- (0,1) node[left, black] {$\boldsymbol{a_2}\hspace{0.3mm}$};
\draw [black, thick,fill=none,->__________] (0,0) node [above left, black] {} -- (0,1) node[above right, black] {$\hspace{0.4mm}V$}; 
\draw [black, thick,fill=none,->__________] (0,0) node [above left, black] {} -- (0,-1) node[below right, black] {$\hspace{0.4mm}V$};
\draw [white, thick,fill=none,->__________] (0,1) node [above left, black] {} -- (0,2) node[left, black] {$\boldsymbol{2a_2}\hspace{0.3mm}$}; 
\draw [black, thick,fill=none,->__________] (0,1) node [above left, black] {} -- (0,2) node[above right, black] {$V''$}; 
\draw [black, thick,fill=none,->__________] (0,-1) node [above left, black] {} -- (0,-2) node[below right, black] {$V''$};
\draw [black, thick,fill=none,->___________________________] (0,0) node [above left, black] {} -- (2,1) node[above right, black] {$V'$};
\draw [black, thick,fill=none,->___________________________] (0,0) node [above left, black] {} -- (-2,-1) node[below left, black] {$V'$};
\draw [white, thick,fill=none,->___________________________] (0,0) node [above left, black] {} -- (2,-1) node[above, black] {$\boldsymbol{a_1}$};
\draw [black, thick,fill=none,->___________________________] (0,0) node [above left, black] {} -- (2,-1) node[below right, black] {$V'$};
\draw [black, thick,fill=none,->___________________________] (0,0) node [above left, black] {} -- (-2,1) node[above left, black] {$V'$};
\node[blue] at (2,0) {};
\node[blue] at (-2,0) {};
\node[blue] at (0,1) {};
\node[blue] at (0,-1) {};
\node[green] at (2,1) {};
\node[green] at (-2,1) {};
\node[green] at (2,-1) {};
\node[green] at (-2,-1) {};
\node[red] at (4,0) {};
\node[red] at (-4,0) {};
\node[red] at (0,2) {};
\node[red] at (0,-2) {};
\node[black] at (4,2) {\scalebox{1.2}{$BC_2$}};

\draw [brown, thin,fill=none,-] (-0.2,-6.3) node [above left, black] {} to [out=-120,in=30](-1.8,-8.1) node[left, black] {$(V'')^*_{K/k} \times (V^{(2)})^*_{K_0/k}$};
\draw [black, thick,fill=none,->_______________________] (0,-6) node [above left, black] {} -- (-2,-6) node[above left, black] {$V\hspace{0.6mm}$};
\draw [black, thick,fill=none,->_______________________] (0,-6) node [above left, black] {} -- (2,-6) node[above right, black] {$\hspace{0.4mm}V$};
\draw [white, thick,fill=none,->__________] (0,-6) node [above left, black] {} -- (0,-5) node[left, black] {$\boldsymbol{a_2}\hspace{0.3mm}$};
\draw [black, thick,fill=none,->__________] (0,-6) node [above left, black] {} -- (0,-5) node[above right, black] {$\hspace{0.4mm}V$}; 
\draw [black, thick,fill=none,->__________] (0,-6) node [above left, black] {} -- (0,-7) node[below right, black] {$\hspace{0.4mm}V$};
\draw [black, thick,fill=none,->___________________________] (0,-6) node [above left, black] {} -- (2,-5) node[above right, black] {$V'$};
\draw [black, thick,fill=none,->___________________________] (0,-6) node [above left, black] {} -- (-2,-7) node[below left, black] {$V'$};
\draw [white, thick,fill=none,->___________________________] (0,-6) node [above left, black] {} -- (2,-7) node[above, black] {$\boldsymbol{a_1}$};
\draw [black, thick,fill=none,->___________________________] (0,-6) node [above left, black] {} -- (2,-7) node[below right, black] {$V'$};
\draw [black, thick,fill=none,->___________________________] (0,-6) node [above left, black] {} -- (-2,-5) node[above left, black] {$V'$};
\node[blue] at (2,-6) {};
\node[blue] at (-2,-6) {};
\node[blue] at (0,-5) {};
\node[blue] at (0,-7) {};
\node[green] at (2,-5) {};
\node[green] at (-2,-5) {};
\node[green] at (2,-7) {};
\node[green] at (-2,-7) {};
\node[black] at (3.5,-4.5) {\scalebox{1.2}{$B_2$}};

\draw [brown, thin,fill=none,-] (-0.2,-12.2) node [above left, black] {} to [out=-140,in=20](-2.8,-13.5) node[left, black] {$(V^{(2)})^*_{K_0/k}$};
\draw [black, thick,fill=none,->_______________________] (0,-12) node [above left, black] {} -- (-2,-12) node[above left, black] {$V\hspace{0.6mm}$};
\draw [white, thick,fill=none,->_______________________] (0,-12) node [above left, black] {} -- (2,-12) node[below, black] {\textcolor{white}{\large T}$\boldsymbol{a_1\!+\! a_2}$\textcolor{white}{T}};
\draw [black, thick,fill=none,->_______________________] (0,-12) node [above left, black] {} -- (2,-12) node[above right, black] {$\hspace{0.4mm}V$};
\draw [white, thick,fill=none,->__________] (0,-12) node [above left, black] {} -- (0,-11) node[left, black] {$\boldsymbol{a_2}\hspace{0.3mm}$};
\draw [black, thick,fill=none,->__________] (0,-12) node [above left, black] {} -- (0,-11) node[above right, black] {$\hspace{0.4mm}V$}; 
\draw [black, thick,fill=none,->__________] (0,-12) node [above left, black] {} -- (0,-13) node[below right, black] {$\hspace{0.4mm}V$};
\node[blue] at (2,-12) {};
\node[blue] at (-2,-12) {};
\node[blue] at (0,-13) {};
\node[blue] at (0,-11) {};
\node[black] at (3.5,-11) {{\scalebox{1.2}{$(A_1)^2$}}};

\node[circle,fill=black,text=black,inner sep=1.4pt] at (3.04,-9.6) {};
\node[circle,fill=black,text=black,inner sep=1.4pt] at (3.28,-9.25) {};
\node[circle,fill=black,text=black,inner sep=1.4pt] at (3.4,-8.85) {};
\node[circle,fill=black,text=black,inner sep=1.4pt] at (-3.04,-9.6) {};
\node[circle,fill=black,text=black,inner sep=1.4pt] at (-3.28,-9.25) {};
\node[circle,fill=black,text=black,inner sep=1.4pt] at (-3.4,-8.85) {};

\node[circle,fill=black,text=black,inner sep=1.4pt] at (3.04,-3.6) {};
\node[circle,fill=black,text=black,inner sep=1.4pt] at (3.28,-3.25) {};
\node[circle,fill=black,text=black,inner sep=1.4pt] at (3.4,-2.85) {};
\node[circle,fill=black,text=black,inner sep=1.4pt] at (-3.04,-3.6) {};
\node[circle,fill=black,text=black,inner sep=1.4pt] at (-3.28,-3.25) {};
\node[circle,fill=black,text=black,inner sep=1.4pt] at (-3.4,-2.85) {};
\end{tikzpicture}
\end{minipage}\hfill
\begin{minipage}[b]{0.6\textwidth}
\caption{Illustration of the structure of the group $G_{K'/k,V^{(2)},V',V''}$ with root system $BC_2$ and Dynkin diagram $\dynkin[scale=2,labels={a_2,\!\!\!\!\!\! a_1,2 a_2}]{A}[2]{oo}$. Each layer indicates a subgroup generated by root groups of certain lengths. Above each root appears the vector space given by the $k$-points of the corresponding root group. Cartan subgroups are also indicated.}\label{BC2}
\end{minipage}
\end{mdframed}
\end{figure}
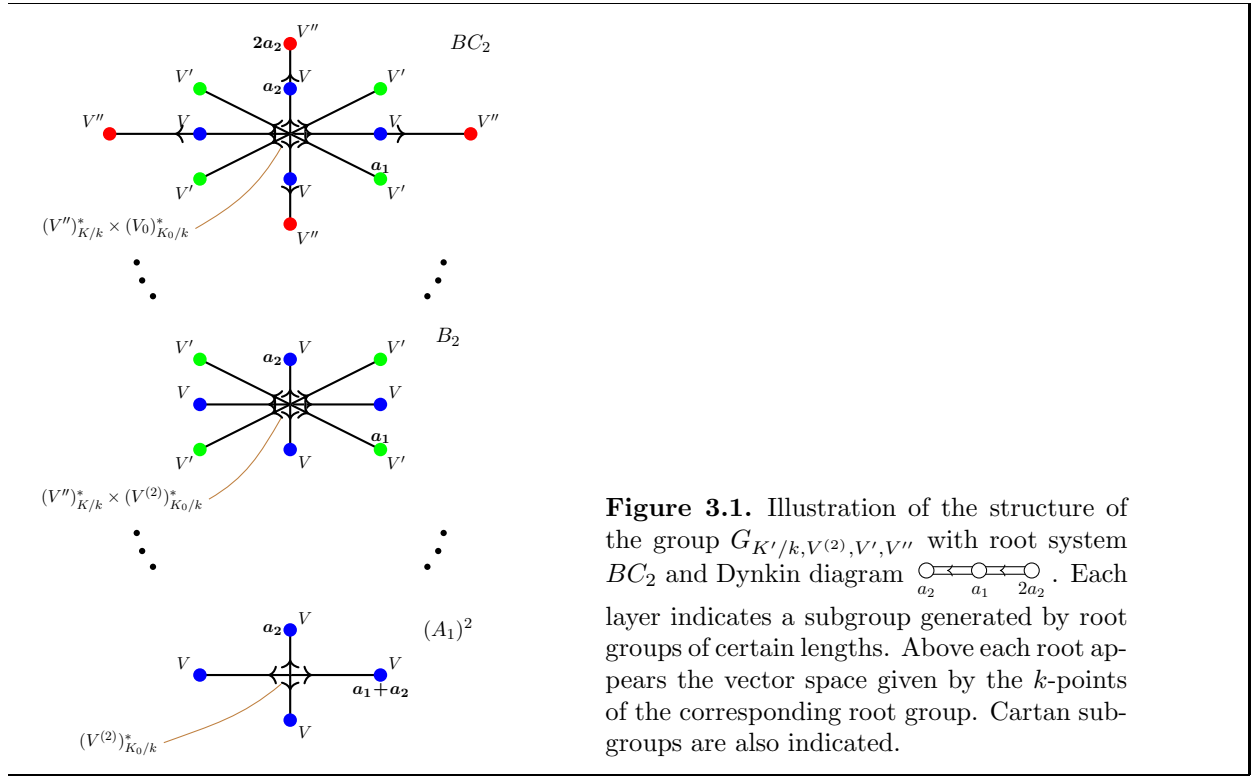

\begin{prop}\label{prop:othertype}  The subgroup 
\[G:=G_{K/k,V^{(2)},V',V''}:=\langle G_{a_1},G_{a_2},G_{2a_2},\mathsf{s}_1,\mathsf{s}_2\rangle\subset\mathsf{Q},\] 
is pseudo-split and absolutely pseudo-simple of minimal type. 
If $H$ is a pseudo-split absolutely pseudo-simple group of minimal type with root system $BC_n$, but $H$ is not isomorphic to the derived subgroup of a group as in Definition \ref{Defn:G}, then $H$ has rank $2$ and there is a tuple $(K/k,V^{(2)},V',V'')$ such that $H\cong G_{K/k,V^{(2)},V',V''}$.
\end{prop}

\begin{proof} Our construction adapts reasonably easily to this case. 

A calculation in $G_{a_1}{}^{\mathsf{s}_1}G_{a_1}G_{a_1}$ shows $\mathsf{s}_1$ is generated by connected subgroups. Similarly for $\mathsf{s}_2$. Thus $G$ is connected and smooth.

Set $G_{a_1+2a_2}:={}^{\mathsf{s}_2}(G_{a_1})$, $G_{a_1+a_2}:={}^{\mathsf{s}_1}(G_{a_2})$ and $G_{2a_1+2a_2}:={}^{\mathsf{s}_1}(G_{2a_2})$. Then define $\UU$ to be the image of the $k$-points of the  multiplication map  \[G_{a_1}\times G_{a_2}\times G_{2a_2}\times G_{a_1+2a_2}\times  G_{2a_1+a_2}\times G_{2a_1+2a_2}\to G.\] To see $\UU$ is a subgroup, we should recheck the commutator calculations of Lemma~\ref{scru}. 
Note that, as before, all commutators between roots can be calculated inside a $B_2$ or a $C_2$ subgroup; in particular, then,
we see that the only two short positive root groups commute, as they are the two positive longer roots in the $B_2$ subsystem
(or, alternatively, use the fact that shorter roots at right angles for $C_2$ in characteristic $2$ commute; it amounts to the same thing).  
Similar considerations allow us to move other roots past each other in a product with less difficulty than in the general case.
Note finally that the requirement that $V''$ be a $K_0$-vector space legitimises the use of the calculation $[x_{2a_2}(t),x_{a_1}(u)]=x_{2a_2+a_1}(tu)\cdot x_{2a_2+2a_1}(tu^2)$ to reorder the $x_\bullet(\bullet)$ in a product of two elements of $\UU$.

With (C6) in mind, we let $\CC$ be the subgroup of $\R_{K/k}(D)(k)$ generated by the elements $h_{a_1}(c/v_1)=s_{a_1}(c)\mathsf{s}_1$, for all $c\in V''\setminus \{0\}$, and $h_{2a_2}(c/v_2)=s_{a_2}(c)\mathsf{s}_2$ for all $c\in V_0\setminus\{0\}$. Then using (C5), we check that $\BB:=\langle \CC,\UU\rangle$ is a semidirect product of $\CC$ and $\UU$. Let also $\NN=\langle \CC,\mathsf{s}_1,\mathsf{s}_2\rangle$. Then the proof that $(\BB,\NN)$ is a $(B,N)$-pair goes through as before, and the proof of Lemma \ref{rootgroupswork} similarly shows $G(k)=\GG$. We claim $\GG$ (hence also $G$) is perfect.

For $a$ a root of $L$, the formula $h_a(c)=s_a(c)s_a(1)$ implies $h_a(c)\in \GG$ for $a=a_1$ or $2a_2$ with $c\in V''$ or $c\in V'\oplus V^{(2)}$ respectively. (For the latter, we use the fact that $\pi$ is an isomorphism on $\CC$.) As $h_a$ is multiplicative we have $h_a(c)\in \GG$ for $c\in k\langle V''\rangle=K^\times$ or $c\in k\langle V'\oplus V^{(2)}\rangle=K_0^\times$, respectively. 

We have $c/v_2\in K_0$ for any fixed $c\in V_0$. As also $V''$ is a $K_0$-subspace, the calculation \[[h_{2a_2}(c/v_2),x_{a_1}(t)]=x_{a_1}((v_2/c+1)t)\] implies we can choose $t$ to reproduce any element $x_a(u)$ with $u\in V''$ as a commutator, so $G_{a_1}(k)\subseteq\D(\GG)$. Similarly, $[h_{a_1}(c/v_1),x_{2a_2}(t)]=x_{2a_2}((c^2/v_1^2+1)t)$, and $V'$ is a $K^2$-subspace of $K$, so $G_{2a_2}(k)\subset \D(\GG)$. 
For the last simple root group, observe $[h_{a_1}(c/v_1),x_{a_2}(t)]=x_{a_2}((c/v_1+1)t)$, and $V$ is a $K$-subspace of $E$, so $G_{a_2}(a)\subset \D(\GG)$. Finally, \begin{align*}[\mathsf{s}_1,x_{a_1}(t)]&=x_{-a_1}\left((v_1^{-2})t\right)x_{a_1}(t),\quad\text{and}\\
[\mathsf{s}_2,x_{2a_2}(t)]&=x_{-2a_2}\left((v_2^{-2})t\right)x_{2a_2}(t).\end{align*}
Substituting $t=v_1$ and $v_2$, respectively, we see $x_{-a_1}(v_1^{-1}),x_{-a_2}(v_2^{-1})\in \D(\GG)$. Thus $\mathsf{s}_1$ and $\mathsf{s}_2$ are in $\D(\GG)$. This proves the claim.

To see $\D(G)$ has a Levi subgroup, we may argue as in Lemma \ref{1inV'}. Let $\tilde h(\mu):=h_{2a_1+2a_2}(\mu)\in D(E)$ for $\mu\in E^\times$ and let it act on $\QQ$ by conjugation. Since the root group $\mathsf{Q}_{2a_1+2a_2}$ commutes with $\mathsf{Q}_{2a_2}$ and $\mathsf{Q}_{a_2}$, $\tilde h(\mu)$ centralises $\mathsf{Q}_{2a_2}$, $\mathsf{Q}_{a_2}$ and the element $\mathsf{s}_2$. On the other hand $(-a_1-2a_2)+(2a_1+2a_2)=a_1$ as a sum of roots and so $\tilde h(\mu)$ scales the $\QQ_{a_1}$ root group by $\mu$ and takes $\mathsf{s}_1$ to $\tilde h_{a_2}(\mu)\mathsf{s}_1$. If $0\neq c\in V''$ then setting $\mu=c^{-1}$, we have $h_{a_2}(\mu)\in {}^{\tilde h(\mu)}\GG$, so $s_2\in {}^{\tilde h(\mu)}\GG$. Furthermore the element $h(\lambda)$ of Lemma \ref{1inV'} lies in $D(E)$ and so commutes with $\tilde h(\mu)$. Replacing $G$ by its conjugate by $h(\lambda)\tilde h(\mu)$ we may assume that the tuple $(K/k,V^{(2)},V',V'')$ has $1\in V'$ and $1\in V''$. In particular $\GG$ contains the elements $x_{a_1}(t),x_{2a_2}(t)$ for $t\in k$,  $\mathsf{s}_1$ and $\mathsf{s}_2$, which together generate a copy of $\Sp_4$. This is sent isomorphically to itself under $\pi$. As per the proof of Lemma \ref{igmap}, this must be a Levi subgroup, and the field of definition of its unipotent radical is $K$.

To verify that $G$ is isomorphic to the group constructed in \cite[\S9.8.3]{CGP15}, just observe that the data $(K/k,V^{(2)},V',V'')$ can be recovered from the field of definition of $\RR_u(G_{\bar k})$ and the root groups of $G$. Then the classification of \cite[Thm.~9.8.6]{CGP15} yields the second sentence.
\end{proof}

\begin{remarks}(i). We have shown that each pseudo-split absolutely pseudo-simple group of minimal type with non-reduced root system contains a split simple subgroup of type $C_n$, specifically a (Levi) subgroup isomorphic to $\Sp_{2n}$. Gopal Prasad has asked us if it also contains a split simple subgroup of type $B_n$. Indeed, if $G$ is constructed from the data $(K/k,V^{(2)},V')$ or $(K/k,V^{(2)},V',V'')$, then suppose $0\neq v\in V\subseteq E$. Then conjugation by $h(v^{-1})$ sends $G\subseteq \QQ$ to a group constructed from analogous data with $1\in V$. Similarly, if $0\neq w\in V''$, then conjugation by the commuting element $\tilde h(w^{-1})$ permits us to assume $1\in V''$. It is easy to see that this modification is enough to guarantee that $G$ contains the canonical copy $M=\langle x_{a_1}(t),\dots,x_{a_n}(t),s_1,\dots, s_n,D\mid t\in k\rangle \cong\SO_{2n+1}\subseteq \mathsf{M}$ .

(ii). In each case, the quotient $\pi(\D(G))$ of $G$ by $\mathsf{U}\cap G$ is also absolutely pseudo-simple of minimal type and it has root system $C_n$: conjugating by $h(\lambda)$ or $h(\lambda)\tilde h(\lambda)$ we may assume it contains the canonical $k$-Levi subgroup $L$ of $\R_{K/k}(L_K)$. Thus it is pseudo-reductive and its root groups are known. One can in fact identify it with one of the groups of type $C$ described in \cite[Thm.~3.4.1(iii)]{CP17}. Essentially, one gets $\pi(\D(G))$ by replacing the long root groups of $\Sp_{2n}(K)$ with the subspace $V^{(2)}\oplus V'$. \end{remarks}

\section{Irreducible representations of \texorpdfstring{$\D(G_{K/k,V^{(2)},V'})$ and $G_{K/k,V^{(2)},V',V''}$}{DG and G}}\label{irreps}

We start with some generalities. Temporarily, let $k$ be an arbitrary field of characteristic $p$ and $G$ a pseudo-split pseudo-reductive $k$-group. Then $G$ has a Levi subgroup $M$ containing a split maximal torus $T$ \cite[Thm~3.4.6]{CGP15}. A choice of a system of positive roots of $M$ determines a subset $X(T)_+\subseteq X(T)$ of dominant weights of $T$. The following is \cite[Thm.~1.2]{BS22}:

\begin{theorem}\label{bs21thm}The isomorphism classes of irreducible representations of $G$ are in $1$--$1$ correspondence with the dominant weights of $T$. If $\lambda\in X(T)_+$ we denote by $L_G(\lambda)$ a corresponding irreducible representation.

On restriction, $L_G(\lambda)$ is $M$-isotypic and semisimple. Furthermore,
\[\dim L_G(\lambda)=\dim L_M(\lambda)\cdot \dim L_C(\lambda).\]\end{theorem}

A general formula for $\dim L_M(\lambda)$ is an exceedingly difficult problem except for low rank, even if we were to assume $p=2$ and $M$ simple of type $C$. However, for the pseudo-simple groups $G$ under consideration in this paper, we give a reduction to this problem by describing $\dim L_C(\lambda)$.

Given a split torus $T$ and a choice of isomorphism $T\cong \Gm^n$, a dominant weight $\lambda\in X(T)_+$ identifies with a homomorphism $(\Gm)^n\to \Gm$; via $(x_1,\dots,x_n)\mapsto \prod x_i^{\lambda_i}$ for certain non-negative integers $\lambda_i$, and we so identify $\lambda$ with the sequence $(\lambda_1,\dots,\lambda_n)$.
 If $(k_1,\dots, k_n)$ is a sequence of finite purely inseparable field extensions of $k$, then denote by $k_i(\lambda_i)$ the subfield $k(k_i^{\lambda_i})=k(k_i^{p^{e_i}})$ of $k_i$, where $e_i=v_p(\lambda_i)$ is the exponent of the highest power of $p$ dividing $\lambda_i$. 
 If $C\cong \prod_{i=1}^n\R_{k_i/k}(\Gm)$ with $T$ identified as a subgroup in the obvious way, then \cite[Thm.~5.8]{BS22} gives the dimension of $L_C(\lambda)$ as $[\kappa:k]$ where $\kappa$ is the compositum of the $k_i(\lambda_i)$ in $\bar k$. We need a slight upgrade of this result.

\begin{theorem}Let $k$ be any field and $C$ a $k$-group with $C\cong C_1\times \dots \times C_n$ where each $C_i$ is a commutative pseudo-split pseudo-reductive $k$-group of rank $1$. Let $T\cong T_1\times \dots\times T_n$ be the maximal split torus of $C$, with each $T_i$ the maximal split torus of $C_i$, and suppose the minimal field of definition of the geometric unipotent radical of $C_i$ is $k_i$. Then for $\lambda\in X(T)$, $L_C(\lambda)$ identifies with the compositum $K$ of the subfields $k_i(\lambda_i)$ of $\bar k$. In particular, $\dim L_C(\lambda)=[K:k]$.\label{prodsofrk1}\end{theorem}

\begin{proof}As usual let $i_{C_i}:C_i\to \R_{k_i/k}(\Gm)$ denote the map corresponding under adjunction to the quotient of $(C_i)_{k_i}$ by its unipotent radical. As each $T_i$ is a Levi subgroup for $C_i$, the map $\prod i_{C_i}:C\to Y:=\prod\R_{k_i/k}(\Gm)$ is an inclusion on restriction to $T$. Consider $L_Y(\lambda)$ as a $C$-module via the map $\prod i_{C_i}$. Now, $L_Y(\lambda)$ is a $\lambda$-weight module for $T$---see \cite[\S5.1]{BS22}---hence as a $C$-module it is isotypic with composition factors isomorphic to $L_C(\lambda)$. Also, $L_Y(\lambda)$ identifies with the compositum of the $k_i(\lambda_i)$ in $\bar k$ by \cite[Thm.~5.8, proof]{BS22}; more specifically, $1\in L_Y(\lambda)$ generates the subfield $k_i(\lambda_i)$ of $K$ under the action of the subgroup $\R_{k_i/k}(\Gm)$. By \cite[Prop.~5.11]{BS22}, $1\in L_Y(\lambda)$ also generates the subfield $k_i(\lambda_i)$ under $C_i$. If $W\cong L_C(\lambda)$ is an irreducible $C$-stable submodule of $L_Y(\lambda)$ containing $w\neq 0$, then because multiplication in the field $L_Y(\lambda)$ commutes with the $Y$-action, multiplying by $w^{-1}$ gives another $C$-stable submodule of $L_Y(\lambda)$, so we may assume $W$ contains $1$. Therefore $W$ also contains the compositum of the $k_i(\lambda_i)$. But this is the whole of $L_C(\lambda)$.\end{proof}

Keep the assumptions as above but return to the case where $p=2$ and $G$ is pseudo-simple and of minimal type with root system $BC_n$ and $K$ the minimal field of definition of the geometric unipotent radical. We treat the following cases in parallel: (i) $G=G_{V^{(2)},V',n}$; (ii) $G=\D(G_{V^{(2)},V',n})$; or if $n=2$, (iii) $G=G_{V^{(2)},V',V''}$. In each case there is a Levi subgroup $M\cong \Sp_{2n}$ containing the split maximal torus $D$. Respectively, we have $C=C_G(D)$ isomorphic to:
\begin{enumerate}
\item $\prod_{i=1}^n \R_{K/k}(\Gm)$; 
\item $\prod_{i=1}^{n-1}\R_{K/k}(\Gm)\times (V_0^*)_{K_0/k}$---\cite[(9.7.6)]{CGP15}; or
\item $(V'')^*_{K/k}\times (V_0)^*_{K_0/k}$---\cite[(9.8.2), Prop.~9.8.4(1)]{CGP15}.\end{enumerate} 
Here $V_0=V^{(2)}\oplus V'$, $K_0=k\langle V_0\rangle$ is the extension of $k$ inside $K$ generated by ratios of nonzero elements from $V_0$, and $(V_0)^*_{K_0/k}$ is the Zariski closure in $\R_{K_0/k}(\Gm)$ of the subgroup generated by ratios of nonzero elements from $V_0$---see \cite[Defn.~9.1.1, $\S$9.7.8]{CGP15} for more discussion. 
We assume that $1\in V'$, which is no loss by Lemma \ref{lem:1inv'} or the proof of Proposition \ref{prop:othertype};
this also ensures that $V_0\subseteq K_0$ in case (iii).
We always take the rank $1$ factors of $C$ in order corresponding to simple roots $\{a_1,\dots,a_{n-1},a_n\}$ of $G$.

If $(k_1,\dots, k_n)$ is a sequence of finite purely inseparable field extensions of $k$, then denote by $k_i(\lambda_i)$ the subfield $k(k_i^{\lambda_i})=k(k_i^{v_p(\lambda_i)})$ of $k_i$, where $v_p(\lambda_i)$ is the highest power of $p=2$ dividing $\lambda_i$. Then \cite[Thm.~5.8]{BS22} gives the dimension of $L_C(\lambda)$ as $[\kappa:k]$ where $\kappa$ is the compositum of the $k_i(\lambda_i)$ in $\bar k$. That establishes the first case of the following:

\begin{theorem} If $\lambda=0$ then $L_G(\lambda)$ is trivial and $\dim L_G(\lambda)=1$. Otherwise:

in case (i), $\dim L_C(\lambda)=\max\{[K(\lambda_i):k]\mid 1\leq i \leq n\}$;

in case (ii), $\dim L_C(\lambda)=\max(\{[K(\lambda_i):k]\mid 1\leq i\leq n-1\}\cup [K_0(\lambda_n):k])$;

in case (iii), $\dim L_C(\lambda)=\max([K(\lambda_1):k],[K_0(\lambda_2):k])$.\end{theorem}

\begin{proof} 
We first observe that the rank $1$ subgroups $(V_0)^*_{K_0/k}$ and $(V'')^*_{K/k}$ have minimal fields of definition of their  unipotent radicals equal to $K_0\subseteq K$ and $K$ respectively.
To see this for $(V_0)^*_{K_0/k}$ note that,  by definition, $(V_0)^*_{K_0/k}\subseteq \R_{K_0/k}(\Gm)$, so the field of definition of its unipotent radical is contained in $K_0$. 
On the other hand, $(V_0)^*_{K_0/k}(k)$ contains generators for $K_0$ as a $k$-algebra, and hence acts irreducibly on $K_0$ through the inclusion $(V_0)^*_{K_0/k}\subseteq \R_{K_0/k}(\Gm)$; this forces the field of definition of its unipotent radical to be at least $K_0$ because that is the field of definition of the unipotent radical of $\R_{K_0/k}(\Gm)$. 
A similar argument works for $(V'')^*_{K/k}$.
Therefore, by Theorem \ref{prodsofrk1}, $L_C(\lambda)$ identifies with the compositum of the $K(\lambda_i)$ for $i\leq n-1$ with either $K(\lambda_n)$ in case (i), or $K_0(\lambda_n)$ in cases (ii) and (iii). Now $K_0$ is sandwiched between $K$ and $kK^2$; thus either some $K(\lambda_i)$ or $K_0(\lambda_n)$ contains all others. The theorem follows.\end{proof}

\begin{remark}In case $n=1$ or $2$ we can be completely explicit about $\dim L_G(\lambda)$ since $\dim L_M(\lambda)$ is so easy to describe.

If $n=1$, then $K=K_0$ and we identify dominant weights with non-negative integers. By Steinberg's Tensor Product Theorem \cite[II.3.17]{Jan03}, a simple module with high weight $\lambda$ is a tensor product of Frobenius twists. If $\lambda=r_0+2r_1+\dots$ with $r_i\in \{0,1\}$ is its $2$-adic expansion, then \begin{equation}\label{stpt}L_M(\lambda)\cong L_M(r_0)\otimes L_M(r_1)^{[1]}\otimes L_M(r_2)^{[2]}\otimes \dots.\end{equation} Each $L(r_i)$ is either trivial or the $2$-dimensional natural module according as $r_i=0$ or $1$. Therefore we obtain that either $\lambda=0$ and $L_G(\lambda)$ is the $1$-dimensional trivial module, or \[\dim L_G(\lambda)=2^{\sum r_i} \cdot [(K)^{2^{j}}:k],\] where $j$ is the minimum integer such that $r_j$ is non-zero.

In case $n=2$, we have $M\cong\Sp_4$ and we can appeal to the exceptional isogeny $\tau:M\to M$ which is a square root of the Frobenius map. Given an $M$-module $V$, one gets another $V^{[\tau]}$ by acting through $\tau$. For $V=L_M(\lambda)$ irreducible, say of high weight $\lambda\in X(D)_+=X(\Gm\times\Gm)_+$, then if $\lambda$ is viewed as a pair of non-negative integers $(a,b)$, we have $V^{[\tau]}$ is another irreducible module, of high weight $\tau^*(\lambda)=(2b,a)$. It is now clear we may write $\lambda$ as a $\tau$-adic expansion $r_0\varpi+\tau^*(r_1\varpi)+(\tau^*)^2(r_2\varpi)+\dots$, where each $r_i\in\{0,1\}$ and $\varpi:=(1,0)$. A version of Steinberg's theorem \cite[\S11]{Ste63} states that \eqref{stpt} still holds, with $\varpi$ appropriately inserted. Since $\varpi$ is the high weight of the natural $4$-dimensional $M$-module, we have $\dim L_M(\lambda)=4^{\sum r_i}$.

Let $G:=G_{V^{(2)},V',V''}$, where we allow $V''=K$. A Cartan subgroup of $G$ is $C=(V'')^*_{K/k}\times (V_0)^*_{K_0/k}$. Note that $K^2\subseteq K_0\subseteq K$. 

Let $r = \sum r_i$.
If $r=0$ then $\lambda = 0$ and $L_G(\lambda)\cong k$ has dimension $1$. 
Otherwise let $j$ be minimal such that $r_j\neq 0$. Then \[\dim L_{G}(\lambda)=\begin{cases}[(K)^{2^{j/2}}:k]\cdot 4^r\text{\quad if $j$ is even;}\\
[(K_0)^{2^{\frac{j-1}{2}}}:k]\cdot 4^r\text{\quad if $j$ is odd.}\end{cases}\]\end{remark}

\appendix\section{An alternative construction}

Following a suggestion of one of the referees, in this short appendix we outline an alternative construction of the groups from Section \ref{defng}.
The idea is to employ \cite[Thm.~C.2.29]{CGP15}, which acts as a ``black box'' producing pseudo-reductive groups from root group data.

We first give a statement of the theorem. 
Let $X$ be a smooth connected affine group over an arbitrary field $k$, and let $S$ be a nontrivial $k$-split torus in $X$. Let $\Delta$ be a non-empty linearly independent subset of the character group $X(S)$ of $S$, and let $C$ be a smooth connected $k$-subgroup of $Z_X(S)$ containing $S$ as a maximal split $k$-torus. For each $a\in \Delta$ suppose there is given a smooth connected $k$-subgroup $F_a$ of $X$ containing $C$ such that $Z_{F_a}(S)=C$ and assume that the set $\Phi(F_a,S)$ of non-zero $S$-weights on $\Lie(F_a)$ contains $\pm a$ and is contained in $\mathbb{Z}\cdot a$. Since $Z_{F_a}(S)=C$, the maximal $k$-split torus $S$ of $C$ is also a maximal $k$-split torus of $F_a$ for all $a\in \Delta$. For $a\in \Delta$, let $E_{\pm a}$ be the root groups of $F_a$ corresponding to $\pm a$. 

\begin{theorem}\label{CGPC.2.29}Assume that for every $a,b\in \Delta$ with $a\neq b$, $E_a$ commutes with $E_{-b}$. For the smooth connected $k$-subgroup $F$ of $X$ generated by the $k$-subgroups $\{F_a\}_{a\in \Delta}$, the following hold:
	\begin{enumerate}
		\item The centraliser $Z_F(S)$ is equal to $C$. So $S$ is a maximal $k$-split torus of $F$ and, if $S$ is a maximal torus of $C$, then it is a maximal torus of $F$.
		\item Any nonzero weight of $S$ on $\Lie(F)$ is either a nonnegative or a nonpositive integral linear combination of elements in $\Delta$.
		\item For each $a\in \Delta$, the $\pm a$-root groups in $F$ are $E_{\pm a}$.
		\item If for every $a\in \Delta$, $F_a$ is quasi-reductive (resp.~pseudo-reductive) then $F$ is quasi-reductive (resp.~pseudo-reductive), and in such cases $\Delta$ is a basis of the root system of $F$ with respect to $S$.
		\item If the $k$-groups $F_a$ are reductive for every $a\in\Delta$ then so is $F$.
		\item If $F_a$ is quasi-reductive for all $a\in \Delta$ then $F$ is functorial with respect to isomorphisms in $(S,\Delta,\{F_a\}_{a\in \Delta})$ in the following sense: if $(X',S',C',\Delta',\{F'_{a'}\}_{a'\in\Delta})$ is a second $5$-tuple such that $E'_{a'}$ commutes with $E'_{-b'}$ for all $a'\neq b'\in \Delta$ and there are a given $k$-isomorphism $f_C:C\cong C'$ restricting to $f_S:S\to S'$ satisfying $X(f_S)(\Delta')=\Delta$ and $k$-isomorphisms $f_{a'}:F_{a'\circ f_S}\cong F'_{a'}$ extending $f_C$ for all $a'\in \Delta'$---so that $F_{a'}$ and hence $F$ is also quasi-reductive---then there is a unique $k$-isomorphism $f:F\cong F'$ extending $f_{a'}$ for all $a'\in \Delta'$.
	\end{enumerate}
\end{theorem}

We can now apply the theorem to the set-up described in the main body of the paper. 
In order to do this, we let $\QQ$ play the role of $X$, $D$ play the role of $S$, $C=\R_{K/k}(D_K)$ and $\Delta=\{a_1,\dots,a_n,2a_n\}$, a base of the full root system $\Phi=\Phi(\QQ,D)$.
Recall that all elements of the Weyl group $W$ of $\Phi$ can be represented by $k$-points of $N_{\QQ}(D)$; in particular, we have the reflections $s_i$ for each root $a_i$.
Let us set $F_{a_1}=\langle G_{a_1},G_{-a_1},C\rangle$, where $G_{-a_1}=G_{a_1}^{s_1}$, and for $2\leq i\leq n-1$, $F_{a_i}=F_{a_1}^w$ where $w\in W$ is such that $w\cdot a_1=a_i$. 
Additionally, put $F_{a_n}=\langle G_{a_n},G_{2a_n},G_{-a_n},G_{-2a_n}\rangle$ where $G_{-a_n}=G_{a_n}^{s_n}$ and $G_{-2a_n}=G_{2a_n}^{s_n}$.

We then wish to show that the smooth and connected $F_{a}$ have $Z_{F_{a}}(S)=C$ and that their root systems with respect to $S$ consist of $\pm a$ or $\pm a, \pm 2a$ if $a=a_n$. 
Once this is achieved, the remarks of Section \ref{pinningq} then imply that the root groups $E_a$ and $E_{-b}$ commute for distinct $a,b\in\Delta$.  Furthermore we need to show that each $F_a$ is pseudo-reductive, implying the same is true for $F$, by part (iv) of the theorem. 

Since the $s_i$ normalise $C$ and $S$, it clearly suffices to show the properties in the previous paragraph for the two cases $a=a_1$ and $a=a_n$. 
In the first case $a=a_1$, we may identify $F_a$ with the Weil restriction $\R_{K/k}(H)$, where $H$ is  a short Levi subgroup of the $K$-group $L_K\cong \Sp_{2m}$ with maximal torus $D_{K}$, and the required properties follow quickly. 

In the second case $a=a_n$, we can proceed as follows: let $F_a=\langle C,G_a,G_{2a},G_{-a},G_{-2a}\rangle$. Let $b\in \{a,2a\}$ and $c$ another root. Then the relations $x_b(v)h_c(u)=h_{c}(u)x_b(u^{\langle b,a^\vee\rangle}v)$ imply that $C(k)$ and hence also $C$ normalises $G_b$. (The interesting case is where $b=2a$, so that $c\in V'$; then $\langle b,a^\vee\rangle$ is always a multiple of $2$, so that $u^{\langle b,a^\vee\rangle}\in K^2$ and its product with $c$ remains in $V'$.) In particular, the root system of $F_a$ is $\{\pm a,\pm 2a\}$.
We also need to show that the root groups of $F_a$ are the $G_b$ and that $C$ is a Cartan subgroup of $F_a$. It will suffice to do so instead for the subgroup $H:=\langle G_a,G_{2a},G_{-a},G_{-2a},C'\rangle\subset \QQ_a$ where $\QQ_a$ is the group generated by the root groups $\MM_{\pm a}$ and $\LL_{\pm a}$; more specifically, to show that the root groups of $H$ are the $G_b$ and that $C':=H\cap C$ is a Cartan subgroup of $H$.

Let $\{e_i,f_i\mid 1\leq i\leq n\}$ be a symplectic basis for $U$ according to the form stabilised by $L$. The centraliser of the subspace $\langle e_2,f_2,\dots, e_n,f_n\rangle$ is a Levi subgroup of $L$ whose derived subgroup $A$ is isomorphic to $\Sp_2\cong \SL_2$ and for which $U'=\langle e_1,f_1\rangle\subseteq U$ affords the structure of a natural module for $A$. Now $\QQ_a=\R_{E/k}(U'\rtimes A)$. Since $A\subseteq \QQ_a$ we have reduced the problem to treating the case $n=1$; that is, it suffices to do our calculations in the special case $\QQ = \QQ_a$. 
The most straightforward way to proceed is by considering $3\times 3$ matrices arising from the embedding of $\QQ$ in a parabolic of $\SL_3$---see the discussion immediately after Proposition \ref{sdpprop}.

Explicitly, from the definition of $M\subseteq Q=U\rtimes L$, the $k$-points of the root groups in $\QQ$ are easily seen to be
\[x_+(t)=\begin{bmatrix}1 & t & 0\\0 & 1 & 0\\0 & 0 & 1\end{bmatrix},\qquad x_-(t)=\begin{bmatrix}1 & 0 & 0\\t & 1 & 0\\0 & 0 & 1\end{bmatrix}\]
\[y_+(t)=\begin{bmatrix}1 & t^2 & 0\\0 & 1 & 0\\0 & t & 1\end{bmatrix},\qquad
y_-(t)=\begin{bmatrix}1 & 0 & 0\\t^2 & 1 & 0\\t & 0 & 1\end{bmatrix}\]
where $\LL=\langle x_+(t),x_-(t)\mid t\in E\rangle$ and  $\MM=\langle y_+(t),y_-(t)\mid t\in E\rangle$. 

Here, the $k$-points of $\mathsf{U}$ identify with the set of matrices
\[\left\{\begin{bmatrix}1 & 0 & 0\\0 & 1 & 0\\t & u & 1\end{bmatrix}\, \middle\vert\, t,u\in E\right\}.\] 
We are therefore led to consider the subgroup 
$$
\mathscr{H}=\langle x_+(t),y_+(u),x_-(t),y_-(u)\mid t\in V',u\in V\rangle\subset H.
$$ 
Note also that the map $\pi$ induces the map on $k$-points
\[\pi(k):\begin{bmatrix}a & b & 0\\c & d & 0\\t & u & e\end{bmatrix}\to \begin{bmatrix}a & b \\ c&d\end{bmatrix}.\]
It is also straightforward to write down the elements of $C(k)$ as diagonal matrices, and to realise the nontrivial element of the Weyl group: it is represented by the element
$$
s=\begin{bmatrix}0 & 1 & 0\\1 & 0 & 0\\0 & 0 & 1\end{bmatrix}\in N_{\QQ}(C)(k).
$$
Now the proofs in the main paper can be carried out with explicit matrix representatives, as follows.
Set $N=\langle s,C\rangle$, $\BB=\langle C,x_+(t),y_+(u)\rangle$, and a general element of $\BB(k)$ is \[\mathsf{b}=
\begin{bmatrix}a &0 &0\\0 &1/a &0\\ 0 &0 &1\end{bmatrix}
\begin{bmatrix}1 &t+u^2 &0\\0 &1 &0\\ 0 &u &1\end{bmatrix}=\begin{bmatrix}a &a(t+u^2) &0\\0 &1/a &0\\ 0 &u &1\end{bmatrix}.\] Now the only calculation remaining is that to show $(\BB(k)s\BB(k))\cdot (\BB(k)s\BB(k)) \subseteq \BB(k) \sqcup \BB(k)s\BB(k)$ which is easily reduced to showing $s\BB(k) s\subseteq \BB(k) s\BB(k)$. Now by matrices:
\[s\mathsf{b}s=
  \begin{bmatrix}0 & 1 & 0\\1 & 0 & 0\\0 & 0 & 1\end{bmatrix}
  \begin{bmatrix}a &a(t+u^2) &0\\0 &1/a &0\\ 0 &u &1\end{bmatrix}
  \begin{bmatrix}0 & 1 & 0\\1 & 0 & 0\\0 & 0 & 1\end{bmatrix}
  =
  \begin{bmatrix}1/a & 0 &0\\a(t+u^2) &a &0\\ u & 0 &1\end{bmatrix}\]
  If $t=u=0$ then $\mathsf{b}\in C(k)$ and hence so is $s\mathsf{b}s={}^s\mathsf{b}\in C(k)$. Otherwise the following check can be made to exhibit $s\mathsf{b}s\in \BB(k)s\BB(k)$:
  \[s\mathsf{b}s = 
  \begin{bmatrix}1/(a(t+u^2)) & 1 &0\\0 & a(t+u^2) &0\\ 0 &u &1\end{bmatrix}
  \begin{bmatrix}0 & 1 & 0\\1 & 0 & 0\\0 & 0 & 1\end{bmatrix}
  \begin{bmatrix}1 & 1/(t+u^2) & 0\\0 & 1 & 0\\0 & u/(t+u^2) & 1\end{bmatrix}
  .\]
  With this established, the argument of Lemma \ref{rootgroupswork} goes through.

We finally note that, in the case $n=2$,  the
additional cases described in Proposition \ref{prop:othertype} can also be constructed inside $\QQ$ using Theorem \ref{CGPC.2.29}.
The key observation is that when $n=2$ we have simple roots $a_1$ (short), $a_2$ (very short), and $2a_2$ (long), and the new groups arise by modifying the subgroups corresponding to the root $a_1$.
In the notation established above, the subgroup $F_{a_1}\cong \R_{K/k}(H)$ has semisimple part isomorphic to $\SL_2$.
Inside a group of the form $\R_{K/k}(\SL_2)$ we can find a $k$-subgroup by generating with the vector subgroup of each root group whose $k$-points are the space $V''$ (these are the subgroups of type $A_1$ denoted by $H_{V''}$ in \cite{CGP15});
this process is described in detail in \cite[pp.\,384--386]{CGP15}.
The upshot is that we can replace $F_{a_1}$ with a different group $F_{a_1}''$ corresponding to the space $V''$, and then $F_{a_2}$ can be dealt with using $3\times 3$ matrices as above, allowing us to apply Theorem \ref{CGPC.2.29} to generate these extra cases.


\subsection*{Conflicts of interest.} The authors declare that they have no conflict of interest.

{\footnotesize
\bibliographystyle{amsalpha}
\bibliography{bib}}

\end{document}